\documentclass[12pt,a4paper,final,reqno]{amsart}%{report}

\newcommand{\includefigure}[2]{\includegraphics[#1]{#2}}
\usepackage{amsmath}
\usepackage{amsthm}
\usepackage{amssymb}
\usepackage{graphicx}
\usepackage{mathrsfs}
\usepackage[english]{babel}
\usepackage[latin1]{inputenc}
\usepackage[T1]{fontenc}
\usepackage{amsrefs}
\usepackage[usenames,dvipsnames]{xcolor}
\usepackage{enumitem}
\usepackage{xcolor}
\usepackage{bbm}
\usepackage{float}
\usepackage{hyperref}
\usepackage{soul}
\usepackage{tikz}
\usetikzlibrary{patterns,decorations.markings}
\usepackage{tikz-cd}
\usepackage{multirow}
\usepackage{booktabs}
\usepackage{subcaption}

\usepackage[makeroom]{cancel}

\numberwithin{equation}{section}

\newtheorem{theorem}{Theorem}[section]
\newtheorem{proposition}[theorem]{Proposition}

\newtheorem{lemma}[theorem]{Lemma}

{\theoremstyle{definition}%{plain}
\newtheorem{remark}[theorem]{Remark}

\newtheorem{defn}[theorem]{Definition}
 }

\newcommand{\pist}{{\pi}} %%Poisson structure%%%
\newcommand{\pimap}{{P}} %%%Poincare map%%%
\newcommand{\cal}{\mathcal}

\newcommand{\CC}{{\cal C}}

\newcommand{\UU}{{\cal U}}

\newcommand{\Nn}{{\mathbb{N}}}

\newcommand{\Rr}{{\mathbb{R}}}
\newcommand{\R}{{\mathbb{R}}}

\newcommand{\Zz}{{\mathbb{Z}}}

\newcommand{\ie}{\textit{i}.\textit{e}., }
\newcommand{\Ker}{\mathrm{Ker}}
\newcommand{\Eq}{\mathrm{Eq}}

\def\diag{\operatorname{diag}}

\def\j{\operatorname{j{}}}
\def\C{\operatorname{C{}}}

\newcommand{\matbrackets}[2]{\left[\, {#1} \right]_{{#2}} }
\newcommand{\unit}{\mathbbm{1}}

\renewcommand{\d}{\mathrm d}

\newcommand{\Mat}{{\rm Mat}}

\newtheorem*{theorem*}{Theorem}
\newcommand{\id}  {\operatorname{id}}

\newcommand{\comment}[1]{}

\newcommand{\Bscr}{\mathscr{B}}

\newcommand{\Dscr}{\mathscr{D}}

\newcommand{\nund}{{\underline{n}}}

\newcommand{\inter}{{\rm int}}

\newcommand{\Poin}[2]{P_{#2}}

\newcommand{\skPoin}[2]{\pi_{#2}}
\newcommand{\skDomPoin}[2]{\Pi_{#2}}
\begin{document}
\selectlanguage{english}
%\title[Hamiltonian Assymptotic Poincare Maps]{Assymptotic Poincare Maps along the %edges of Polytopes (II)\\
 %\textnormal{\tiny  Hamiltonian systems}}
 
\title[Hamiltonian Polymatrix Replicators]{Asymptotic Dynamics of Hamiltonian Polymatrix Replicators}

% \setkomafont{subtitle}{\normalfont\Large}
%\date{May 2014}
\date{\today}  
\author[Alishah]{Hassan Najafi Alishah}
\address{ Departamento de Matem\'atica, Instituto de Ci\^encias Exatas\\
Universidade Federal de Minas Gerais \\
31123-970 Belo Horizonte\\
MG - Brazil }
\email{halishah@mat.ufmg.br}

\author[Duarte]{Pedro Duarte}
\address{Departamento de Matem\'atica and CMAF \\
Faculdade de Ci\^encias\\
Universidade de Lisboa\\
Campo Grande, Edificio C6, Piso 2\\
1749-016 Lisboa, Portugal 
}
\email{pduarte@fc.ul.pt}

\author[Peixe]{Telmo Peixe}
\address{ISEG-Lisbon School of Economics \& Management\\
Universidade de Lisboa\\
REM-Research in Economics and Mathematics\\
CEMAPRE-Centro de Matem\'atica Aplicada \`a Previs\~ao e Decis\~ao Econ\'omica\\
Lisboa, Portugal.
}
\email{telmop@iseg.ulisboa.pt}

\begin{abstract}
In a previous paper~\cite{ADP2020} we have studied flows defined on polytopes, presenting a new method to encapsulate its asymptotic dynamics along  the edge-vertex heteroclinic network.
These results apply to the class of polymatrix replicator systems, which contains several important models in Evolutionary Game Theory.
Here we establish the Hamiltonian character of the asymptotic dynamics of Hamiltonian polymatrix replicators.
\end{abstract}

\maketitle
%\tableofcontents
\addtocontents{toc}{\protect\setcounter{tocdepth}{1}}
%to not appear subsection titles in the table of contents

%%%%%%%%%%%%%%%%%%%%%%%%%%%  Introduction  %%%%%%%%%%%%%%%%%%%%%%%%%%%%%%%%%%%%%%%%%%%%%%%%%

\section{Introduction}
\label{sec:intro}

A new method to study the asymptotic dynamics of flows defined on polytopes
was presented in~\cite{ADP2020}.
This method allows us to analyze the asymptotic dynamics of flows defined on polytopes along
the edge-vertex heteroclinic network.
Examples of such dynamical systems arise naturally in the context of Evolutionary Game Theory (EGT) developed by J. Maynard Smith and G. R. Price~\cite{smith1973logic}.

One such example is the \textit{polymatrix replicator}, introduced in~\cite{AD2015,ADP2015}.
This is a system of ordinary differential equations  that models
the time evolution of behavioral strategies of individuals in a stratified population.

The polymatrix replicator induces a flow on a prism (simple polytope) given by a finite product of simplices.
These systems extend the class of the replicator and the bimatrix replicator equations studied e.g. in~\cite{TJ1978} and~\cite{schuster1981coyness,schuster1981selfregulation}, respectively.

In~\cite{AD2015} the authors have introduced the subclass of
conservative polymatrix replicators (see Definition~\ref{CPR}) which are Hamiltonian systems with respect to appropriate Poisson structures.
In~\cite{ADP2015} these Hamiltonian polymatrix replicators are used to describe the asymptotic dynamics of the larger class of dissipative polymatrix replicators.

For  Hamiltonian vector fields on symplectic manifolds it is well known that  the Poincar\'e map preserves  the induced symplectic structure on any transversal section. In this paper we extend  this  fact to Hamiltonian systems on Poisson manifolds,  showing that any transversal section inherits a Poisson structure and  the Poincar\'e map preserves it.
Using this result we  study the Hamiltonian character of the asymptotic dynamics of conservative polymatrix replicators along their edge-vertex heteroclinic network.
Our main result states that for conservative polymatrix replicators the map describing the asymptotic dynamics is Hamiltonian with respect to an appropriate Poisson structure (Theorem~\ref{main theorem}).

The paper is organized as follows. 
In Section~\ref{sec:outline} we recall the method in~\cite{ADP2020}, outlining the construction of the asymptotic dynamics for a large class of flows on polytopes that includes the polymatrix replicators.
In Section~\ref{sec:poisson-poincare} we define Poincar\'e maps for Hamiltonian systems on Poisson manifolds.
In Section~\ref{sec:polymatrix-replicators} we  provide a short introduction to polymatrix replicators, following~\cite{AD2015}. Namely, we state the basic definitions and results for the class of conservative polymatrix replicators, that we also  designate as \textit{Hamiltonian polymatrix replicators}.
In Section~\ref{sec:asymptotic-polymatrix-replicator} we review the main definitions and results for the polymatrix replicator systems regarding the construction outlined in Section~\ref{sec:outline}.
In Section~\ref{sec:hamiltonian-character} we analyze the Hamiltonian character of Poincar\'e maps in the case of Hamiltonian polymatrix replicators. 
Finally, in Section~\ref{sec:example} we present an example of a five-dimensional Hamiltonian polymatrix replicator to illustrate the main concepts and results of this paper. The graphics of this section were produced with \textit{Wolfram Mathematica} and \textit{Geogebra} software.

%%%%%%%%%%%%%%%%%%%%%%%%%%%%%%%%%%%%%%%%%%
\section{Outline of the construction}
\label{sec:outline}
%\input{outline-construction.tex}

%Asymptotic dynamics, along the edges, of flows on polytopes
We now outline the construction of the asymptotic dynamics
for a large class of flows on polytopes that includes the 
polymatrix replicators.
A {\em polytope} is a compact convex set in some Euclidean space obtained as the intersection of finitely many half-spaces.
A polytope is called {\em simple} if the number of edges (or facets) incident with each vertex equals the polytope's dimension.
The phase space of polymatrix replicators, that are prisms given by a finite product of simplices, are examples of simple polytopes.
In~\cite{ADP2020} we consider analytic vector fields defined on simple polytopes which have the property of being  tangent to every face of the polytope.
Such vector fields induce complete flows on the polytope
 which leave all faces  invariant.
Vertices of the polytope are singularities of the vector field,
while edges without singularities, called {\em flowing edges}, consist of single orbits flowing between two end-point vertices.
The  vertices and flowing edges form a heteroclinic network of the vector field. The purpose of this construction is to analyze the asymptotic dynamics  of the vector field along this one-dimensional skeleton. Throughout the text we assume that every vector field is
{\em non-degenerate}. This means that the transversal derivative of the vector field is never identically zero along any facet of the polytope.

%% Poincare Maps
The  analysis of the vector field's dynamics  along  its edge-vertex heteroclinic network makes use of Poincar\'e maps  between cross sections tranversal to the flowing edges.
Any Poincar\'e map along a heteroclinic or homoclinic orbit  is a composition of two types of maps,  global  and local Poincar\'e maps. 
A {\em global map},  denoted by $P_\gamma$, is defined in a tubular neighborhood of any flowing-edge $\gamma$. It maps points between two cross sections $\Sigma_\gamma^-$ and $\Sigma_\gamma^+$ transversal to the flow along the edge $\gamma$. 
A {\em local map}, denoted by $P_v$, is defined in a neighborhood of any vertex $v$. For any pair of flowing-edges  $\gamma,\gamma'$ 
such that $v$ is both the ending point of $\gamma'$ and the starting point of $\gamma$, the local map $P_v$ takes points from $\Sigma_{\gamma'}^+$
to  $\Sigma_{\gamma}^-$ (see Figure~\ref{local-global-poincare}).

\begin{center}
	\begin{figure}[h]
		\includegraphics[width=6cm]{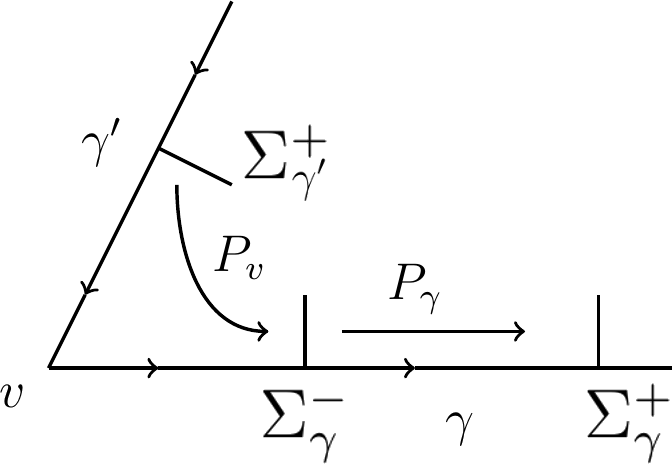}
		\caption{\footnotesize{Local and global  
			Poincar\'e maps along a heteroclinic orbit.}} \label{local-global-poincare}
	\end{figure}
\end{center}

% Skeleton character
Asymptotically, the nonlinear character of the  global Poincar\'{e} maps fade away as we approach  a heteroclinic orbit. This means that these non-linearities are irrelevant for the asymptotic analysis.
For regular\footnote{The reader should bare in mind that the concept of regularity used here (Definition~\ref{def:regular}) is more restrictive then the one in~\cite[Definition $6.3$]{ADP2020}.} vector fields, the  {\em skeleton character}  at a vertex,
defined as the set of eigenvalues of the tangent map along the edge eigen-directions, completely determines the asymptotic behavior of
the local Poincar\'e map at that vertex.

\begin{center}
	\begin{figure}[h]
		\includegraphics[width=10cm]{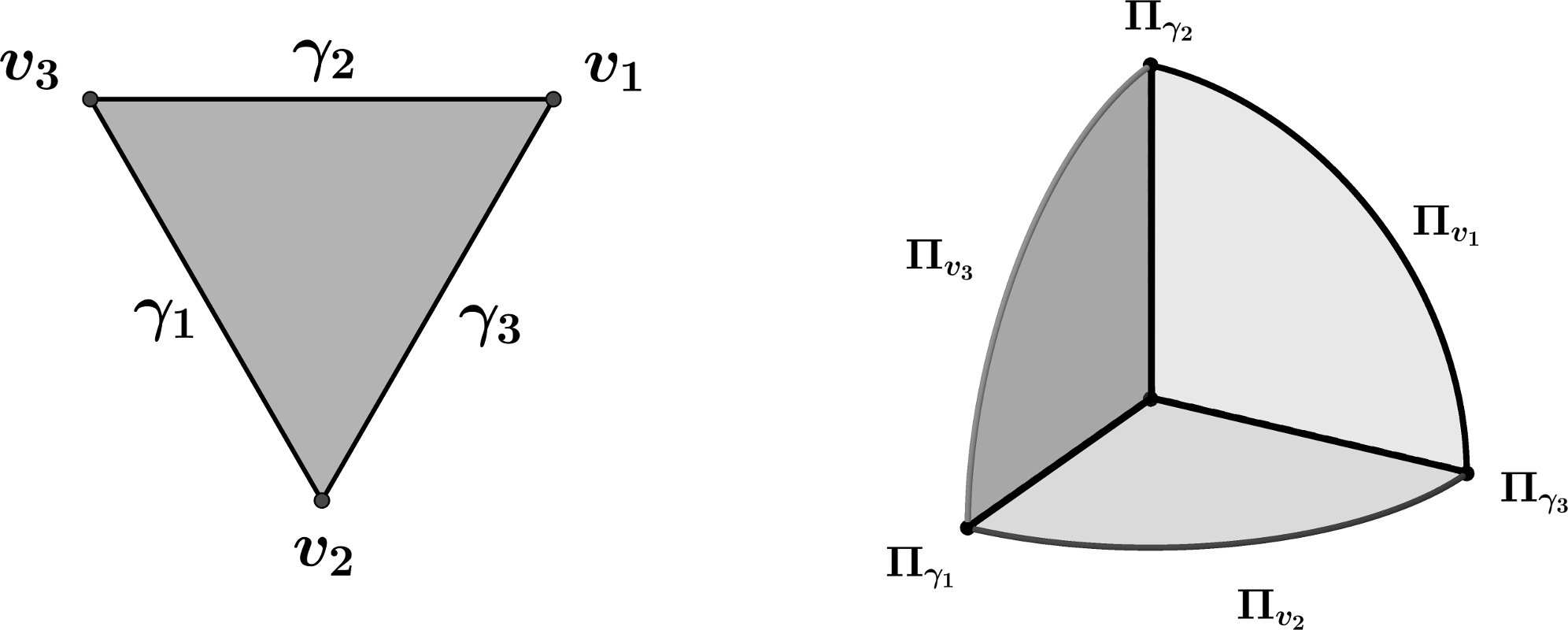}
		\caption{\footnotesize{Dual cone of a triangle in $\Rr^{F}$.}} \label{dualcone1}
	\end{figure}
\end{center}

%% The dual cone 
To describe the limit dynamical behavior we introduce   the {\em dual cone} of a polytope where the asymptotic piecewise linear dynamics  unfolds.
This space lies inside $\Rr^{F}$, where $F$ is the set of the polytope's facets. 
The dual cone of a $d$-dimensional simple polytope $\Gamma$ is the union 
$$\CC^\ast(\Gamma) :=\bigcup_{v\in V}\Pi_v\;, $$
where for each vertex $v$,
$\Pi_v$ is the $d$-dimensional sector consisting
of points $y\in\R^F$ with non-negative coordinates such that
$y_\sigma=0$ for every facet that does not contain $v$.
See Figure~\ref{dualcone1}.

Given a vector field $X$  on a 
$d$-dimensional polytope $\Gamma\subset \Rr^d$, we now describe a rescaling change of coordinates $\Psi_\epsilon^X$, depending on a blow up parameter $\epsilon$. See Figure~\ref{asym-linearisation}.

\begin{center}
	\begin{figure}[h]
		\includegraphics[width=11cm]{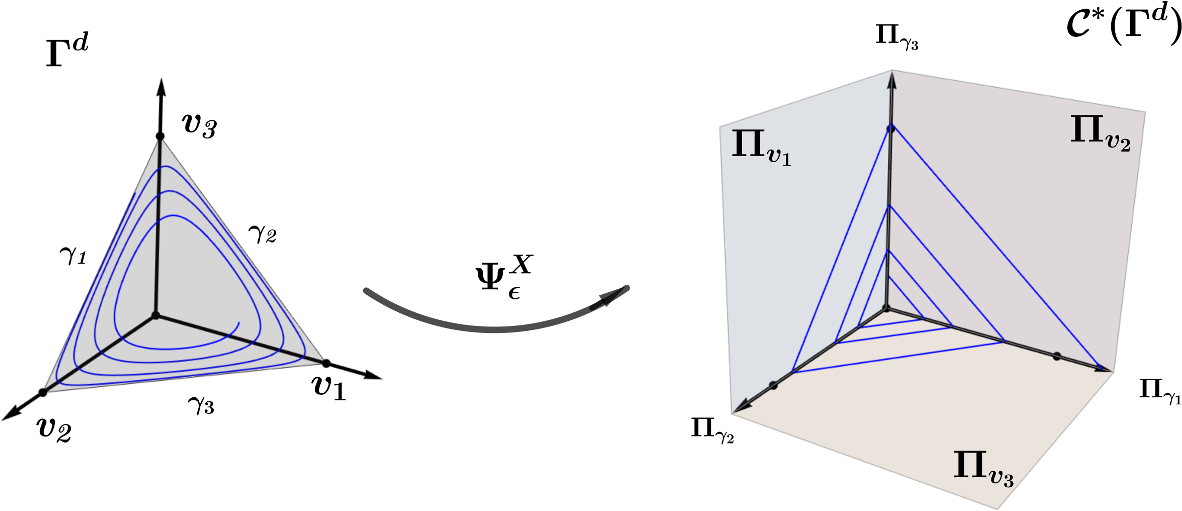}
		\caption{\footnotesize{Asymptotic linearisation on the dual cone.
			The left image represents an orbit on the simplex $\Delta^2$ and the right one
			the corresponding (nearly) piecewise linear image under the map  $\Psi_\epsilon^X$ on the dual cone.}
		}\label{asym-linearisation} 
	\end{figure}
\end{center}

This change of coordinates  maps  tubular neighborhoods of edges and vertices to  the dual cone  $\CC^\ast(\Gamma)$. 
For instance,  the tubular neighborhood $N_v$ of a vertex $v$ is defined as follows.
Consider a system $(x_1,\ldots, x_d)$  of affine coordinates around $v$,
which assigns   coordinates $(0,\ldots, 0)$ to $v$ and such that the hyperplanes $x_j=0$ are precisely the facets of the polytope through $v$. 
Then $N_v$ is defined by
$$ N_v:=\{ p\in\Gamma^d\colon  0\leq x_j(p) \leq 1 \,\text{ for }\, 1\leq j\leq d  \} . $$
The sets $\{x_j=0\}\cap N_v$  are
called the outer facets of $N_v$. The remaining facets of $N_v$, defined by equations like $x_i=1$, are called the inner facets of $N_v$.
The previous cross sections $\Sigma_\gamma^\pm$  can be chosen to match these inner facets of the neighborhoods  $N_v$.

%% The rescaling coordinates
The rescaling change of coordinates $\Psi^X_\epsilon$
maps  $N_v$  to  the sector $\Pi_v$.
Enumerating $F$  so that the  facets through $v$ are precisely   $\sigma_1,\ldots, \sigma_d$,  
the map $\Psi^X_\epsilon$ is defined on the neighborhood $N_v$  by
$$ \Psi^X_\epsilon(q):=(-\epsilon^2\,\log x_1(q), \ldots,
-\epsilon^2\,\log x_d(q),0,\ldots, 0  ) . $$
Similarly, given an edge $\gamma$,   $\Psi^X_\epsilon$ maps
a tubular neighborhood $N_\gamma$  of  $\gamma$ to  the facet sector $\Pi_\gamma:=\Pi_v\cap \Pi_{v'}$ of $\Pi_v$ where $v'$ is the other end-point of $\gamma$. The map  $\Psi^X_\epsilon$  sends interior facets of $N_v$ and $N_\gamma$ respectively to boundary facets of
$\Pi_v$ and $\Pi_\gamma$ while it  maps outer facets 
of $N_v$ and $N_\gamma$   to infinity. 
%% The skeleton vector field
As the rescaling parameter $\epsilon$ tends to $0$, 
the rescaled push-forward $\epsilon^{-2}\,(\Psi^X_\epsilon)_\ast X$ of the vector field $X$ converges to a constant vector field $\chi^v$  on each sector $\Pi_v$. This means that asymptotically, as $\epsilon\to 0$, trajectories become lines in the coordinates   $(y_\sigma)_{\sigma\in F}=\Psi^X_\epsilon$.
Given a flowing-edge $\gamma$ between vertices $v$ and $v'$, the map $\Psi^X_\epsilon$ over $N_\gamma$ depends only on the  coordinates 
transversal to $\gamma$. Moreover,
as  $\epsilon\to 0$ the global Poincar\'e map $P_\gamma$ converges to the identity map in the coordinates $(y_\sigma)_{\sigma\in F} =\Psi^X_\epsilon$.
Hence the sector
$\Pi_\gamma$ is  naturally  identified as the common facet between the sectors $\Pi_v$ and $\Pi_{v'}$.
 Hence the asymptotic dynamics along the edge-vertex heteroclinic network is completely determined by the vector field's geometry at the vertex singularities and can be 
described by a piecewise constant vector field $\chi$ on the dual cone, whose components are precisely those of the skeleton character of  $X$. We  refer to this piecewise constant vector field as the {\em skeleton vector field} of $X$.  This  vector field $\chi$ induces a piecewise linear flow on the dual cone   whose dynamics can be  computationally explored.

We use  Poincar\'e maps for a global analysis of the asymptotic dynamics of the flow of $X$.
We consider a subset $S$ of flowing-edges with the property  that every  heteroclinic cycle goes through at least one edge in $S$.
Such sets are called structural sets.
The flow of $X$ induces a Poincar\'e  map $P_S$
on the system of cross sections $\Sigma_S:= \cup_{\gamma\in S}\Sigma^+_\gamma$. Each branch of the Poincar\'e  map $P_S$ is associated with a  heteroclinic path  starting with an edge in $S$ and ending at its first return to another edge in $S$. These heteroclinic paths are the branches of $S$. The flow of the skeleton vector field $\chi$ also induces a  first return map
$\pi_S:D_S\subset \Pi_S\to\Pi_S$ on the system of cross sections $\Pi_S:= \cup_{\gamma\in S}\Pi_\gamma$. This map $\pi_S$,
called the \textit{skeleton flow map}, is piecewise linear and its domain is a finite union of open convex cones.
In some cases, see Proposition~\ref{partition}, the map
 $\pi_S$ becomes a closed dynamical system.

%% The main result
We can now recall  
	the main result in \cite{ADP2020}, Theorem~\ref{asymp:main:theorem} below, which says that  under the rescaling change of coordinates
	$\Psi^X_\epsilon$, the 
	Poincar\'e map $P_S$ converges in the $C^\infty$ topology to
	the skeleton flow map $\pi_S$, in the sense that the following limit holds
	$$ \lim_{\epsilon\to 0 }  \Psi^X_\epsilon \circ \Poin{X}{S}\circ
	(\Psi^X_\epsilon)^{-1} = \pi_S  $$
	with uniform convergence of the map and its derivatives over any compact set contained in the domain $D_S\subset \Pi_S$.

%% Invariant functions
Consider now, for each facet $\sigma$ of the polytope, an affine function $\Rr^d\ni q\mapsto x_\sigma(q)\in\Rr$ which vanishes on $\sigma$  and is strictly positive on the rest of the polytope. With this family of affine functions we can present the polytope  as 
$\Gamma^d=\cap_{\sigma\in F} \{  x_\sigma \geq 0  \}$.
Any function function $h:\inter(\Gamma^d)\to\Rr$ of the form
$$h(q)=\sum_{\sigma\in F} c_\sigma\, \log x_\sigma(q)\quad (c_\sigma\in\Rr) $$
rescales to the following  piecewise linear function on the dual cone  
$$\eta(y):=\sum_{\sigma\in F} c_\sigma\, y_\sigma $$
in the sense that
$\eta = \lim_{\epsilon\to 0} \epsilon^{-2}\, (h\circ
(\Psi^X_\epsilon)^{-1} ) $.
When all coefficients $c_\sigma$ have the same sign
then $\eta$ is a proper function on the dual cone and  all levels of $\eta$ are compact sets.
If the function $h$ is invariant under the flow
of  $X$, i.e.
$h\circ P_S=h$, then the piecewise linear function
$\eta$ is also invariant under the skeleton flow,
i.e. $\eta\circ \pi_S=\eta$.
Thus integrals of motion (of vector fields on polytopes) of the previous form give rise to (asymptotic) piecewise linear  integrals of motion for the skeleton flow.

%%%%%%%%%%%%%%%%%%%%%%%%%%%%%%%%%%%%%%%%%%
\section{Poisson Poincar\'e maps}
\label{sec:poisson-poincare}

In this section we will define Poincar\'e map for Hamiltonian systems   on  Poisson manifolds. For  Hamiltonian vector fields on symplectic manifolds it is well known that  the Poincar\'e map preserves  the induced symplectic structure on any transversal section (see~\cite{MR2554208}*{Theorem 1.8.}). We extend  this  fact to Hamiltonian systems on Poisson manifolds,  showing that any transversal section inherits a Poisson structure and  the Poincar\'e map preserves this structure.

A Poisson manifold is a pair $(M,\pi)$ where $M$ is a smooth manifold without boundary and $\pist$ a Poisson structure on $M$.
Recall that a Poisson structure is a smooth bivector field $\pist$ with the property that $[\pist,\pist]=0$, where $[\cdot, \cdot]$ is the Schouten bracket  (cf. e.g.~\cite{MR2178041}).
The bivector field $\pi$ defines a vector bundle map 
\begin{equation}\label{eq:pi-sharp} \pi^\sharp\colon T^\ast  M\rightarrow TM\quad\mbox{by}\quad \xi\to\pist(\xi,.).
\end{equation}
The image of this map is an integrable singular distribution which integrates to a symplectic foliation, i.e., a foliation whose leaves have a symplectic structure induced by the Poisson structure. 

Notice that a Poisson structure can also be defined as a Lie bracket $\{\cdot,\cdot\}$ on $C^\infty(M)\times C^\infty(M)$ satisfying the Leibniz rule
$$
\{f,gh\}=\{f,g\}h+g\{f,h\},
\qquad
f,g,h\in C^\infty(M).
$$
These two  descriptions are related by $\pist(\d f,\d g)=\{f,g\}$.
In a local coordinate chart $(U,x_1,..,x_n)$, or equivalently  when $M=\Rr^n$,  a Poisson bracket takes the form
\[\{f,g\}(x)=(\d_x f)^t
\matbrackets{\pi_{ij}(x)}{ij} \d_x g,\]
where  $\pi(x)=\matbrackets{\pi_{ij}(x)}{ij}= \matbrackets{\{x_i,x_j\}(x)}{ij}$   is 
a skew symmetric matrix valued smooth function satisfying 
\begin{equation*}\label{eq:condition:Poisson}
\sum_{l=1}^{n}  \frac{\partial\pi_{ij}}{\partial x_l}\pi_{lk}+\frac{\partial\pi_{jk}}{\partial x_l}\pi_{li}+\frac{\partial\pi_{ki}}{\partial x_l}\pi_{lj} =0\quad\quad \forall i,j,k\;.
\end{equation*}

\begin{defn}\label{def:Poissonmap}
Let $(M,\{,\}_M)$ and $(N,\{.,.\}_N)$ be two Poisson manifolds. A smooth map $\psi:M\to N$ will be called a {\em Poisson map} iff
\begin{equation}\label{eq:Poissoncondition2-1}
\{f\circ\psi,h\circ\psi\}_M=\{f,h\}_N\circ\psi\quad \forall f,h\in C^\infty(N).
\end{equation}
Using the map $\pi^\sharp$, defined in \eqref{eq:pi-sharp}, this condition reads as
\begin{equation}\label{eq:Poissoncondition2}
(d\psi) \pi_M^\sharp (d\psi)^\ast=\pi_N^\sharp \circ \psi ,
\end{equation}
where we use the notation $(d\psi)^\ast$ to denote the adjoint operator of  $d\psi$.
Notice that, if $d\psi$ is the Jacobian matrix of $\psi$ in local coordinates, then the matrix representative of the pullback will be $(d\psi)^t$. 
\end{defn}
\begin{remark}\label{remark:push-pull-Poisson}
When $\psi$ is a diffeomorphism and only one of the manifolds $M$ or $N$ is Poisson manifold, Definition~\ref{def:Poissonmap} can be used to push-forward or pullback the Poisson structure to the other manifold. 
\end{remark}
\begin{defn}
Let $(M,\pi)$ be a Poisson manifold. The Hamiltonian vector field associated to a given function $H:M\to\mathbb{R}$  is defined by derivation $X_H(f):=\{H,f\}$ \, for $f\in C^\infty(M)$, or equivalently $X_H:=\pi^\sharp(dH)$. 
\end{defn}
As in the symplectic case, to define the Poincar\'e map we will consider the traversal sections inside the level set of the Hamiltonian. 
We will show that such a transversal section is a cosymplectic submanifolds of the ambient Poisson manifold and naturally inherits a Poisson structure. For more details on cosymplectic submanifolds see  \cite{CALVO2010259}*{Section $5.1$}. 
\begin{defn}[]\label{cosymplectic}
$N\subset (M,\pi)$ is a {\it cosymplectic submanifold} if it is the level set of {\it second class constraints} i.e., 
$N=\cap_{i=1}^{2k} G_i^{-1}(0)$ where  $\{G_1,...,G_{2k}\}$ are functions such that    $[\{G_i,G_j\}(x)]_{i,j}$ is an invertible matrix at all points $x\in N$.
\end{defn}

\begin{remark}
A constraint is called first class if it Poisson commutes with other constraints of the  system. Sometimes, in the literature, a constraint that has non-zero Poisson bracket with at least one other constraint of the system is called a second class constraint. Definition~\ref{cosymplectic} demands a stronger condition, but the cosymplectic submanifolds that we will use have codimension $2$, where having non-zero Poisson bracket with the other constraint  is the same as  $[\{G_i,G_j\}(x)]_{i,j=1,2}$ being an invertible matrix.
\end{remark}

Every, cosympletic submanifold is naturally equipped with a Poisson bracket called Dirac bracket. Paul Dirac, \cite{paul-dirac-book}, developed this bracket to treat classical systems with second class constraints in Hamiltonian mechanics.

\begin{defn}\label{dirac-bracket} For cosymplectic submanifold  $N\subset(M,\pi)$, let $${G_1,..,G_{2k}:U\to\mathbb{R}}$$ be  its second class constraints,  where $U$ is a small enough neighborhood of $N$ in $M$ such that the matrix  $ [\{G_i,G_j\}(x)]_{i,j}$ is invertible at all points $x\in U$. The {\it Dirac bracket} is defined on $\C^\infty(U)$ by 
 \begin{equation}\label{eq:dirac-bracket}
 \{f,g\}_{\rm Dirac}=\{f,g\}-[\{f,G_i\}]^t\,[\{G_i,G_j\}]^{-1}\,[\{G_i,g\}],
 \end{equation}
 where $[\{.,G_i\}]$ is the column matrix with components $\{.,G_i\}$, $i=1,\ldots,2k$. 
\end{defn}

Dirac bracket is actually a Poisson bracket on the open submanifold $U$, see \cite{CALVO2010259}. It takes an easy calculation to see that constraint functions $G_i,\,i=1,\ldots,2k$ are Casimirs of Dirac bracket.  This fact allows the restriction of Dirac bracket  to the cosymplectic submanifold $N$. Note that, in general, restricting (pulling back) a Poisson structure to an arbitrary submanifold is not straightforward. Actually, the decomposition 
\begin{equation}\label{eq:decomposition}
\pi^\sharp(T_xN^\circ)\oplus T_xN=T_x M
\end{equation}
that holds for every point $x\in N$ and a strait forward calculation yield the independence of the extension in the following definition.  In Equation~\eqref{eq:decomposition}, the term $T_xN^\circ$  is the annihilator of $T_xN$ in $T^\ast_xM$. We will use this notation in the rest of the paper. 
Equation~\eqref{eq:decomposition} can be used as definition of a cosymplectic submanifold, see~\cite{CALVO2010259}, but for our propose it suits better to use the second class constraints to define cosymplectic submanifolds. 
\begin{defn}\label{restricted-dirac-bracket}
The restricted Dirac bracket on cosymplectic submanifold $N$, which will be also referred  to as {\it Dirac bracket,} is simply defined by extending {\it in any arbitrary way} functions on $N$ to functions on $U$, calculating their Dirac bracket on $U$ and restricting the result back to $N$.    
\end{defn}

We consider a Hamiltonian $H$ on the $m$-dimensional Poisson manifold $(M,\pist)$ and its associated Hamiltonian vector field defined by $X_H=\{H,.\}=\pi^\sharp(dH)$. For a given point $x_0\in M$ let  $U$ be a neighborhood around it such that $X_H(x)\neq0\quad \forall x\in U$, and  $\mathcal{E}_{x_0}$ be the energy surface passing through $x_0$, i.e., the connected component of $H^{-1}(H(x_0))$ containing $x_0$. We call {\em level transversal section} to $X_H$ at a regular point $x_0\in M$ any $(m-2)$-dimensional transversal section $\Sigma\subset \mathcal{E}_{x_0} \cap U$ through $x_0$.

The following lemma shows that $\Sigma$ is a cosymplectic submanifold.  
\begin{lemma}\label{lemma:cosymplectic}
 Every level transversal section $\Sigma$ is a cosymplectic submanifold of $M$.
\end{lemma}
\begin{proof}
Since $d_{x_0} H\neq0$, there exist a function $G$ locally defined in $U$ (shrink $U$ if necessary) and linearly independent from $H$ such that 
$$\Sigma=\mathcal{E}_{x_0} \cap U\cap G^{-1}(G(x_0)).$$
 Then, we have 
\[
 \pist(\d H,\d G)= X_H(G)= dG(X_H) \neq 0  
\]
by transversality. This finishes the proof. 
\end{proof}

\begin{remark}\label{set-tilde-notation}
The second term in the right hand side of Equation~\eqref{eq:dirac-bracket} is 
\begin{equation*}%\label{second-term-of-dirac-bracket}
\begin{bmatrix}\{f,H\}\\\{f,G\}\end{bmatrix}^t\begin{bmatrix}0&\{H,G\}\\\{G,H\}&0\end{bmatrix}^{-1}\begin{bmatrix}\{H,g\}\\\{G,g\}\end{bmatrix}.\
\end{equation*}
Then, using extensions $\tilde{f}$ and $\tilde{g}$ of $f,g\in C^\infty(\Sigma)$ such that at every point $x\in\Sigma$ \emph{their differentials vanish on $X_H$}, yields
\[\{f,g\}_{\rm Dirac}= \{\tilde{f},\tilde{g}\}|_{\Sigma}.\]
We will use this fact to simplify our proofs but arbitrary extensions are more suitable for calculating the Dirac structure. 
\end{remark}

We will use the same notation $f$ for arbitrary extension and reserve the notation $\tilde {f}$ for extension that their differentials vanishes on $X_H$ at every point $x\in\Sigma$. To avoid any possible confusion, we observe that in~\cite{MR2128714}*{Section 8} and~\cite{MR2128714}*{Section 8} the notation $\tilde{f}$ is used in a slightly different sense.

\begin{remark}\label{remark:Poisson-dirac}
Cosymplectic submanifolds are special examples of the so called Poisson-Dirac submanifolds, see~\cite{MR2128714}*{Section 8}. The induced Poisson structure on a Poisson-Dirac submanifold is defined by using extensions such that their differentials vanish on $\pi^\sharp(T\Sigma^\circ)$. In~\cite{MR2128714}*{Section 8} and~\cite{CALVO2010259}*{Lemma $5.1$} the notation $\tilde{f}$ is used for this type of extensions. For a cosymplectic submanifold $\Sigma$ given by second class constraints $G_1,...,G_{2k}$, we have
\begin{equation}\label{eq:annihilator}
\pi^\sharp(T\Sigma^\circ)=\oplus_{i=1}^{2k}\mathbb{R} X_{G_i},
\end{equation}
and  the Dirac bracket coincides with the bracket induced in this way, see  \cite{CALVO2010259}*{Section $5.1$}. In our case, we only have two constraints $H,G$ and requiring the vanishing of the differential only on  $X_H$ (or $X_G$) at every point $x\in\Sigma$ is enough to obtain the same induced Poisson bracket. 
\end{remark}

For a fixed time $t_0$, let $x_1=\phi_H(t_0,x_0)$, where $\phi_H$ is the flow of the Hamiltonian vector field $X_H$, and $\Sigma_0, \Sigma_1$ be level transversal sections at $x_0$ and $x_1$, respectively. 
As usual, a Poincar\'e map $\pimap=\phi_H(\tau(x),x)$ can be defined from an appropriate neighborhood of $x_0$ in $\Sigma_0$ to a neighborhood of $x_1$ in $\Sigma_1$. The existence of the smooth function $\tau(x)$ is guaranteed by the Implicit Function Theorem. We replace $\Sigma_0$ and $\Sigma_1$ by the domain and the image of the Poincar\'e map $P$.

By Lemma~\ref{lemma:cosymplectic} both $\Sigma_i$, $i=0,1$, are  cosymplectic submanifolds equipped with Dirac brackets 
$\{.,.\}_{{\rm Dirac}_i},\, i=0,1$. We will show that the Poincar\'e map $\pimap$ is a Poisson map (see Definition~\ref{def:Poissonmap}). 

\begin{proposition}\label{pois-poin-map}
 The Poincar\'e map 
 $$\pimap:(\Sigma_0,\{.,.\}_{{\rm Dirac}_0})\rightarrow(\Sigma_1,\{.,.\}_{{\rm Dirac}_1})$$
  is a Poisson map.  
\end{proposition}

\begin{proof}
We define $\tilde{P}:U_0\to U_1$ by
$$
\tilde{\pimap}(x):=\phi_H(\tilde{\tau}(x),x),
$$
 where $\tilde{\tau}$ is an extension of $\tau$ to a neighborhood $U_0$ of $x_0$ such that its differential, $d\tilde{\tau}$, vanishes on $X_H$. Both neighborhood $U_0$ and $U_1$ can be shrunk, if necessary,  in a way that both Dirac brackets around $\Sigma_0$ and $\Sigma_1$ are defined in $U_0$ and $U_1$, respectively. A straightforward calculation shows that  for every point $x$ in the domain of $\tilde{\tau}$, we have
 \[
  d_{x}\tilde{\pimap}=d_{x}\phi_H^{\tilde{\tau}(x)}+(d_{x}\tilde{\tau})X_H(\phi_H^{\tilde{\tau}(x)}(x)),
 \]
 where $\phi_H^{\bar{\tau}(x)}(.)=\phi_H(\tilde{\tau}(x),.)$. Furthermore, for every $x\in \Sigma_0$, we have 
\begin{align}\label{eq:preserve-xh}
d_{x}\tilde{P}(X_H(x))&=d_{x}\phi_H^{\tilde{\tau}(x)}(X_H(x))+(\underbrace{d_{x}\tilde{\tau}(X_H(x))}_{=0})X_H(\phi_H^{\tilde{\tau}(x)}(x))\notag\\
&=d_{x}\phi_H^{\tilde{\tau}(x)}(X_H(x))=X_H(\phi_H^{\tilde{\tau}(x)}(x))=X_H(\tilde{P}(x)).
\end{align}
Note that $d_{x}\phi_H^{\tilde{\tau}(x)}$ in Equation~\eqref{eq:preserve-xh} is the derivative of time-$\tilde{\tau}(x)$ flow of $X_H$ and the fixed time flow maps of $X_H$ are  Poisson maps, i.e. it sends Hamiltonian vector fields to Hamiltonian vector field. Furthermore,  the flow of $X_H$ preserves $H$, this means that 
$$H(\tilde{P}(x))=H(\phi_H^{\tilde{\tau}(x)}(x))=H(x).$$
As we set in Remark~\ref{set-tilde-notation}, let $\tilde{f}$ be an extension of a given $f\in C^\infty(\Sigma_1)$ such that $$d_x\tilde{f}(X_H)=0,\quad\forall x\in\Sigma_1,$$
then for every $x\in\Sigma_0$, 
\begin{align*}
d_x(\tilde{f}\circ\tilde{P})(X_{H})&=d_{\tilde{P}(x)}\tilde{f}\circ d_x\tilde{P}(X_H)=d_{\tilde{P}(x)}\tilde{f}(X_H)=0.
\end{align*}
Now, for  $f,g\in C^\infty(\Sigma_1)$ and $x\in\Sigma_0$ we have 
 \begin{align*}
 &\{\tilde{f}\circ\tilde{\pimap},\tilde{g}\circ\tilde{\pimap}\}(x)=\pist_x\left((d_x\tilde{\pimap})^\ast\d_x\tilde{f},(d_x\tilde{\pimap})^\ast\d_x\tilde{g},\right)\\
 &= \pist_x\left((d_{x}\phi_H^{\bar{\tau}(x)})^\ast\d_x\tilde{f}+(d_{x}\bar{\tau}X_H)^\ast\d_x\tilde{f},(d_{x}\phi_H^{\bar{\tau}(x)})^\ast\d_x\tilde{g}+(d_{x}\bar{\tau}X_H)^\ast\d_x\tilde{g}\right)\\
 &= \pist_x\left((d_{x}\phi_H^{\bar{\tau}(x)})^\ast\d_x\tilde{f},(d_{x}\phi_H^{\bar{\tau}(x)})^\ast\d_x\tilde{g}\right)\\
 &= \pist(\d\tilde{f},\d\tilde{g})(\tilde{\pimap}(x))=\{\tilde{f},\tilde{g}\}(\tilde{\pimap}(x)),
 \end{align*}
 and consequently, 
 \begin{align*}
  \{f\circ \pimap,g\circ \pimap\}_{{\rm Dirac}_0} &=\{\tilde{f}\circ\tilde{\pimap},\tilde{g}\circ\tilde{\pimap}\}|_{\Sigma_0}\\
  &=\{\tilde{f},\tilde{g}\}\circ\tilde{\pimap}|_{\Sigma_0}\\
  &=\{f,g\}_{{\rm Dirac}_1}\circ \pimap,
 \end{align*}
 where we used Remark~\ref{set-tilde-notation}. This finishes the proof.
\end{proof}

%%%%%%%%%%%%%%%%%%%%%%%%%%%
\section{Hamiltonian polymatrix replicators}
\label{sec:polymatrix-replicators}

In this  section we  provide a short introduction to polymatrix replicators, following~\cite{AD2015}. In particular we will focus on the class of conservative polymatrix replicators that we designate as \textit{Hamiltonian polymatrix replicators}.

Consider a population divided in $p$ groups where each group is labeled by an integer $\alpha\in \{1, \dots ,p \}$, and the 
individuals of each group $\alpha$ have exactly $n_\alpha$ strategies to interact with other members of the population (including of the same group).
In total we have  $n=\sum_{\alpha=1}^p n_\alpha$  strategies that we label
by the integers $i\in \{1,\ldots, n\}$, denoting by
$$[\alpha]:=\{n_1+\cdots+n_{\alpha-1}+1,  \,  \ldots, \,  n_1+\cdots+ n_\alpha\}\subset \Nn  $$
the set (interval) of strategies of group $\alpha$.

Given $\alpha,\beta\in\{1,\ldots,p\}$, consider a real $n_\alpha\times n_\beta$ matrix, say $A^{\alpha,\beta}$, whose entries $a_{ij}^{\alpha,\beta}$, with $i\in [\alpha]$ and $j\in [\beta]$, represent the payoff of an individual of the group $\alpha$ using the $i^{\textrm{th}}$ strategy when interacting with an individual of the group $\beta$ using the $j^{\textrm{th}}$ strategy.
Thus the matrix  $A$ with entries $a_{ij}^{\alpha,\beta}$, where $\alpha,\beta\in\{1,\ldots,p\}$, $i\in [\alpha]$ and $j\in [\beta]$,  is a square matrix of order $n=n_1+\ldots + n_p$, consisting of the block matrices $A^{\alpha,\beta}$.

Let $\nund=(n_1,\ldots, n_p)$.
The \emph{state} of the population is described by a point $x=(x^ \alpha)_{1\le \alpha\le p}$ in the \emph{polytope}
$$ \Gamma_{\nund} :=\Delta^{n_1-1}\times \ldots \times \Delta^{n_p-1}\subset \R^n \;, $$
where $\Delta^{n_\alpha-1}=\{x\in\R_+^{[\alpha]}: \sum_{i\in [\alpha]} x_i^\alpha=1\}$,
$x^\alpha=(x_i^\alpha)_{i\in [\alpha]}$ 
and the entry $x_i^\alpha$ represents the  usage frequency of the $i^{\textrm{th}}$ strategy within the group $\alpha$.
We denote by $\partial\Gamma_{\nund}$ the boundary of $\Gamma_{\nund}$.

Assuming random encounters between individuals, for each group $\alpha\in\{1,\dots ,p\}$,
the average payoff of a strategy $i\in [\alpha]$ within a population with state $x$  is given by
$$ (Ax)_{i}=\sum_{\beta=1}^p \left( A^{\alpha,\beta} \right)_i x^\beta 
= \sum_{\beta=1}^p \sum_{k\in [\beta]}  a_{ik}^{\alpha,\beta}x_k^\beta \,, $$
where  the overall average payoff of group $\alpha$ is given by
$$ \sum_{i\in [\alpha]} x_i^{\alpha} \left( Ax \right)_{i} \,.$$

Demanding that the logarithmic growth rate  of the frequency of each strategy $i\in [\alpha]$, $\alpha\in \{1,\dots ,p\}$, is equal
to the payoff difference between strategy $i$ and the overall average payoff of group $\alpha$ yields the
system of ordinary differential equations defined on the polytope $\Gamma_{\nund}$,
\begin{footnotesize}
\begin{equation}\label{eq:ode:pmg}
\frac{ d x_i^\alpha}{dt}  = x_i^\alpha  \left( (Ax)_{i} - \sum_{i\in [\alpha]}x_i^{\alpha}\left( Ax \right)_{i}\right), \,\,\alpha\in\{1,\dots ,p\}, 
i\in [\alpha],
\end{equation}
\end{footnotesize}
that will be designated as a \emph{polymatrix replicator}.

If $p=1$ equation~\eqref{eq:ode:pmg} becomes the usual replicator equation with payoff matrix $A$.
When $p=2$ and  $A^{11}=A^{22}=0$ are null matrices, equation~\eqref{eq:ode:pmg}  becomes the  bimatrix replicator equation with payoff matrices $A^{12}$ and $(A^{21})^t$.

The flow $\phi_{\nund,A}^t$ of this equation leaves the polytope $\Gamma_{\nund}$ invariant.
The proof of this fact is analogous to that for  the bimatrix replicator equation, see~\cite[Section 10.3]{HS1998}.
Hence, by compactness of $\Gamma_{\nund}$, the flow $\phi_{\nund,A}^t$ is complete.
From now on the term \emph{polymatrix replicator} will also refer to the flow $\phi_{\nund,A}^t$ and the
underlying vector field on $\Gamma_{\nund}$, denoted by $X_{\nund, A}$.

Given $\nund=(n_1,\ldots, n_p)$, let
$$\mathscr{I}_{\nund}:=\{\,I\subset \{1,\ldots, n\}\,:\,
\# \left(I\cap [\alpha]\right) \geq 1,\; \forall\,\alpha=1,\ldots, p\,\}\;.$$
A set $I\in \mathscr{I}_{\nund}$ determines the facet $\sigma_I:=\{\,x\in\Gamma_{\nund}\,:\,
x_j=0,\, \forall\, j\notin I\,\}$ of $\Gamma_{\nund}$.
The correspondence between labels in $\mathscr{I}_{\nund}$
and facets of $\Gamma_{\nund}$ is bijective. 

\newcommand{\isigma}{\sigma^\circ}

\begin{remark}\label{stratification-of-prism}
The partition of $\Gamma_{\nund}$ into the interiors  $\isigma_I:=\inter(\sigma_I)$, with $I\in\mathscr{I}_{\nund}$, is a smooth stratification of $\Gamma_{\nund}$ with strata $\isigma_I$. Every stratum $\isigma_I$ is a  connected open submanifold and for any pair $\isigma_{I_1},\,\isigma_{I_2}$ if $\isigma_{I_1}\cap \sigma_{I_2}\neq\emptyset$ then $\sigma_{I_1}\subset \sigma_{I_2}$. For more on smooth stratification see~\cite{FO} and references therein. 
\end{remark}

For a set $I\in\mathscr{I}_{\nund}$  consider the pair $(\nund^I,A_I)$,
where
${\nund^I=(n_1^I,\ldots, n_p^I)}$ with $n_\alpha^I=\#(I\cap [\alpha] )$,  and  $A_I=\matbrackets{a_{ij}}{i,j\in I}$.  
\begin{proposition}~\cite[Proposition $3$]{AD2015}\label{invariance}
Given $I\in\mathscr{I}_{\nund}$, the facet $\sigma_I$ of $\Gamma_{\nund}$  is invariant under the flow of $X_{\nund,A}$ and the restriction of~\eqref{eq:ode:pmg} to $\sigma_I$ is the polymatrix replicator $X_{\nund^I,A_I}$.
\end{proposition}

For a fixed $\nund=(n_1,\ldots, n_p)$ the correspondence $A\mapsto X_{\nund,A}$ is linear and its kernel consists  of the matrices $C=\left( C^{\alpha,\beta}\right)_{1\le\alpha,\beta\le p}$ where each block $C^{\alpha,\beta}$ has equal rows, i.e., has the form
\[C^{\alpha,\beta}=\begin{pmatrix}
c^{\alpha,\beta}_1&c^{\alpha,\beta}_2&\ldots&c^{\alpha,\beta}_n\\c^{\alpha,\beta}_1&c^{\alpha,\beta}_2&\ldots&c^{\alpha,\beta}_n\\\vdots&\vdots&&\vdots\\c^{\alpha,\beta}_1&c^{\alpha,\beta}_2&\ldots&c^{\alpha,\beta}_n
\end{pmatrix}.
\]
Thus  ${X_{\nund,A_1}=X_{\nund,A_2}}$  if and only if  for every $\alpha,\beta\in\{1,\dots,p\}$ the matrix
$A^{\alpha,\beta}_1-A^{\alpha,\beta}_2$ has equal rows (see~\cite[Proposition 1]{AD2015}).

We have now the following characterization of the interior equilibria.

\begin{proposition}~\cite[Proposition $2$]{AD2015}
\label{prop int equilibria}
Given a polymatrix replicator $X_{\nund,A}$, a point
$q\in {\rm int}(\Gamma_\nund)$ is an equilibrium of $X_{\nund,A}$ \, iff \,
$(A\,q)_i=(A\,q)_j$ for all $i,j\in[\alpha]$ and 
$\alpha=1,\ldots, p$.

In particular the set of interior equilibria of $X_{\nund,A}$ is the intersection of some affine subspace with ${\rm int}(\Gamma_\nund)$.
\end{proposition}

\begin{defn}\label{CPR}
A polymatrix replicator  $X_{\nund,A}$ is said to be {\em conservative}  if there exists:
\begin{enumerate}
\item[(a)] a point $q\in\Rr^n$, called {\em formal equilibrium}, such that 
$(A\,q)_i=(A \, q)_j$ for all $i,j\in[\alpha]$, and all $\alpha=1,\ldots, p$ and $\sum_{j\in[\alpha]} q_j=1$;
\item[(b)] matrices $A_0, D\in\Mat_{n\times n}(\Rr)$ such that 
\begin{enumerate}
\item[(i)] $X_{\nund,A_0D}=X_{\nund,A}$,
\item[(ii)] $A_0$ is a skew symmetric, and
\item[(iii)] $D=\diag(\lambda_1 I_{n_1},\dots, \lambda_p I_{n_p})$  with $\lambda_\alpha\neq 0$
for all $\alpha\in\{1,\dots, p\}$.
\end{enumerate}
\end{enumerate}
The matrix $A_0$ will be referred to as a {\em skew symmetric model}  for $X_{\nund,A}$, and $(\lambda_1,\ldots,\lambda_p)\in(\Rr^ \ast)^p$   as a {\em scaling co-vector}.
\end{defn}

In~\cite{ADP2015}, another characterization  of conservative polymatrix replicators, using quadratic forms, is provided.
Furthermore, in~\cite{hassan-JGM-2020} the concept of conservative replicator equations (where $p=1$) is generalized  using Dirac structures.

In what follows, the vectors in $\Rr^n$, or $\Rr^{[\alpha]}$, are identified with column vectors.
Let $\unit_{n} =(1,..,1)^t\in\Rr^{n}$. We will omit the subscript $n$
whenever the dimension of this vector is clear from the context.
Similarly, we write $I=I_n$ for the $n\times n$ identity matrix. 
Given $x\in \Rr^n$, we denote by $D_x$ the $n\times n$ diagonal matrix
$D_x := \diag(x_1,\dots ,x_n)$.
For each $\alpha\in\{1,\ldots, p\}$ we define the $n_\alpha\times n_\alpha$ matrix
$$ T^\alpha_x:= x^\alpha\, \unit^t -I\;,$$
and $T_x$ the $n\times n$ block diagonal matrix
$T_x := \diag(T^1_x,\dots ,T^p_x)$.

\bigskip
 
Given an anti-symmetric matrix $A_0$, we 
define the  skew symmetric matrix valued  mapping $\pi_{A_0}:\Rr^n\to \Mat_{n\times n}(\Rr)$  
\begin{equation}\label{eq:Poissonpmg}
\pi_{A_0}(x):=(-1)\, T_x\, D_x \, A_0\, D_x\,T_x^t.
\end{equation}

The interior of the polytope $\Gamma_\nund$, denoted by ${\rm int}(\Gamma_\nund)$, equipped with $\pi_{A_0}$ is a Poisson manifold, see~\cite{AD2015}*{Theorem 3.5}. Furthermore, we have the following theorem.

\begin{theorem}~\cite{AD2015}*{Theorem $3.7$}\label{conservative:hamiltonian}
Consider a conservative polymatrix replicator $X_{\nund,A}$ with
formal equilibrium $q$, skew symmetric model $A_0$ and scaling co-vector $(\lambda_1,\ldots, \lambda_p)$.
Then $X_{\nund,A}$, restricted to ${\rm int}(\Gamma_\nund)$, is Hamiltonian with Hamiltonian function
\begin{equation}\label{eq:constant:of:motion}
h(x) =\sum_{\beta=1}^p\lambda_\beta\sum_{j\in [\beta]}q^{\beta}_j\log x_j^{\beta}\;.
\end{equation}
\end{theorem}

%%%%%%%%%%%%%%%%%%%%%%%%%%%%%%%%%%%%%%%%%%%%%%%%%%%
\section{Asymptotic dynamics}
\label{sec:asymptotic-polymatrix-replicator}

Given a  polymatrix replicator $X_{\nund,A}$, the edges and vertices of the polytope $\Gamma_\nund$ form a (edge-vertex) heteroclinic network for the associated flow.
In this section we recall the technique developed in~\cite{ADP2020}  to analyze the asymptotic dynamics of a flow on a polytope along its heteroclinic edge network. In particular we review the main definitions and results for the  polymatrix replicator  $X_{\nund,A}$  on the polytope $\Gamma_\nund$. 

The affine support of $\Gamma_\nund$ is the smallest affine subspace of $\R^n$ that contains $\Gamma_{\nund}$. It is the subspace
$E=E_1\times\ldots\times E_p$ where for $\alpha=1,\ldots,p$,
$$E_\alpha:=\left\{x^\alpha\in\R^{[\alpha]}\, : \, \sum_{i\in [\alpha]} x^\alpha_i=1\right\}.$$

Following~\cite[Definition 3.1]{ADP2020} we introduce a {\em defining family} for the polytope $\Gamma_{\nund}$.
The affine functions ${\{f_i:E\to\mathbb{R}\}_{1\le i\le n}}$ where $f_i(x)=x_i$, form  a defining family for $\Gamma_{\nund}$ because they satisfy:
\begin{itemize}
\item[(a)]$\displaystyle\Gamma_{\nund}=\bigcap_{i\in I} f_i^{-1}([0,+\infty[)$,
\item[(b)] $\displaystyle\Gamma_{\nund}\cap f_i^{-1}(0)\neq\emptyset$ for all $i\in\{1,\dots,n\}$, and
\item[(c)] given $J\subseteq \{1,\ldots,n\}$ such that $\displaystyle\Gamma_{\nund}\cap\left(\bigcap_{j\in J}f_j^{-1}(0)\right)\neq\emptyset$, the linear $1$-forms $(\d f_j)_p$ are linearly independent at every point $\displaystyle p\in\cap_{j\in J}f_j^{-1}(0)$.
\end{itemize}

\newcommand{\Vset}{\mathcal{V}_{\nund}}
\newcommand{\Eset}{\mathcal{E}_{\nund}}
\newcommand{\Fset}{\mathcal{F}_{\nund}}

Next we introduce convenient labels for vertices, facets  and edges
of $\Gamma_{\nund}$. Let $(e_1,\ldots, e_n)$ be the canonical basis of $\R^n$ and denote by $\Vset$ the Cartesian product 
$\Vset := \prod_{\alpha=1}^p[\alpha]$  which contains $ \prod_{\alpha=1}^p n_\alpha$ elements. Each label
$\j=(j_1,\ldots, j_p)\in\Vset$ determines the vertex
$v_{\j}:=(e_{j_1},...,e_{j_p})$ of $\Gamma_{\nund}$.
This labeling is one-to-one. 
The set $\Fset:=\{1,2,\ldots, n\}$ can be used to label the
$n$ facets of $\Gamma_{\nund}$. Each integer $i\in \Fset$ labels the facet $\sigma_i:= \Gamma_\nund\cap\{x_{i}=0\}$   of
$\Gamma_{\nund}$. Edges can be labeled by the set
$ \Eset:=\left\{ J\in \mathscr{I}_{\nund}  \, \colon \;
\#J=p+1  \; \right\}$.
Given $J\in \Eset$ there exists a unique (unordered) pair of labels
$\j_1, \j_2\in \Vset$ such that 
$J$ is the union of the strategies in $\j_1$ and $\j_2$.
The label $J$ determines the edge
$\gamma_J:=\{ t v_{\j_1}+ (1-t) v_{\j_2} \colon 0\leq t \leq 1\}$. Again the correspondence $J\mapsto \gamma_J$ between
labels $J\in\Eset$ and edges of $\Gamma_{\nund}$ is one-to-one.
 
Given   a vertex $v$ of $\Gamma_{\nund}$, we denote by $F_v$ and $E_v$ respectively the sets of facets and edges of $\Gamma_\nund$ that  contain $v$.
Given $\j=(j_1,\ldots, j_p)\in \Vset$  
$$F_{v_{\j}}=\{\sigma_{i}\, \colon\, i\in\Fset \setminus \{j_1,\ldots, j_p\}   \} $$
and this set of facets contains exactly $n-p=\dim(\Gamma_{\nund})$ elements.

Triples in
$$C:=\{\, (v,\gamma,\sigma)\in V\times E\times F\,\colon\,  \gamma\cap \sigma= \{v\}\, \},$$
are called {\em corners}. Any pair of elements in a corner  uniquely determines the third one.  Therefore,  sometimes we will shortly refer to a corner $(v,\gamma,\sigma)$ as $(v,\gamma)$ or $(v,\sigma)$. An edge $\gamma$ with end-points $v,v'$ determines two corners $(v,\gamma,\sigma)$ and $(v^\prime,\gamma,\sigma')$, called the
{\em end corners of} $\gamma$.  The facets $\sigma,\sigma'$ are referred to as the {\em opposite facets of} $\gamma$. 

\begin{remark}
\label{local coordinate system}
In a small neighborhood of a given vertex $v=v_{\j}$, where $\j=(j_1,\ldots ,j_p)\in \Vset$, the affine functions $f_{k}:\Gamma_\nund\to\mathbb{R}$, $f_{k}(x):=x_{k}$, with $k\in \Fset\setminus \{j_1,\ldots, j_p\}$, can be used as a coordinate system for $\Gamma_\nund$.  
\end{remark}

 Given a polymatrix replicator $X_{\nund,A}$ and a facet $\sigma_{i}$ with $i\in [\alpha]$, $\alpha\in \{1,\ldots, p\}$,  the $i^{\rm th}$ component of $X_{\nund,A}$ is given by
$$\d f_{i}(X_{\nund,A})= x_{i}\,\left( (A\,x)_{i} - \sum_{\beta=1}^p (x^{\alpha})^ T A^{\alpha,\beta} x^{\beta} \right).$$

	A polymatrix replicator $X_{\nund,A}$ is called \emph{non-degenerate} if for any $i\in \Fset$, the  function $H_i:\Gamma_{\nund}\to\R$, 
	$$H_i(x):= f_i(x)^{-1}\, \d f_{i}(X_{\nund,A}(x)) = (A\,x)_{i} - \sum_{\beta=1}^p (x^{\alpha})^ T A^{\alpha,\beta} x^{\beta}  $$
	 is not identically zero along $\sigma_i$.

Clearly generic polymatrix replicators are non-degenerate. 
Using the concept of  order of a vector field along a facet~\cite[Definition $4.2$]{ADP2020},  $X_{\nund,A}$ is   non-degenerate if and only if all facets of $\Gamma_{\nund}$ have order $1$.
From now on we will only consider non-degenerate polymatrix replicators.

\begin{defn}  
	 \label{skeletoncharacter}
	The \emph{skeleton character} of  polymatrix replicator 
	$X_{\nund,A}$ is defined to be the matrix
	$\chi:=(\chi^v_\sigma)_{(v,\sigma_{k_\alpha})\in V\times F}$ where 
	\begin{equation*}
	\chi^v_\sigma:=
	\left\{\begin{array}{ccc}-H_\sigma(v), &\quad& v\in \sigma \\
	0  &\quad& \text{otherwise}\end{array}\right. 
	\end{equation*}
	where $H_\sigma$ stands for $H_i$ when $\sigma=\sigma_i$ with $i\in\Fset$.
	For a fixed vertex $v$, the vector $\chi^v:= (\chi^v_\sigma)_{\sigma\in F}$ is referred to as the \emph{skeleton character} at $v$.
\end{defn}

\begin{remark} Given a corner $(v,\gamma,\sigma)$ of $\Gamma_{\nund}$,	$H_\sigma(v)$ is the eigenvalue of the tangent map $(\d X_{\nund,A})_v$
	along the eigen-direction parallel to $\gamma$.
\end{remark}

\begin{proposition} \label{prop:skeletoncharacter}
If $X_{\nund,A}$ is a non-degenerate polymatrix replicator for every vertex $v=v_{\j}$ with label  
	$\j=(j_1,\dots,j_p)\in\Vset$, and every facet $\sigma=\sigma_i$ with $i\in \Fset$ and $i\in [\alpha]$
	the skeleton character of $X_{\nund,A}$ is given by 
	\begin{align*}
	\chi^{v}_{\sigma}=\left\{\begin{array}{cc}
	\vspace{3mm}
	\sum_{\beta=1}^p(a_{j_\alpha j_\beta}-a_{i j_\beta})&\text{if} \; v \in \sigma \\
	0 &     \text{otherwise} \;.
	\end{array}\right.
	\end{align*}
\end{proposition}

\begin{proof}
Straightforward calculation.
\end{proof}

\begin{remark}\label{sign-character}
For a given corner $(v,\gamma,\sigma)$ of $\Gamma_{\nund}$, 
\begin{itemize}
	\item if $\chi^v_\sigma<0$   then $v$  is the $\alpha$-limit of an orbit in $\gamma$, and
	\item if $\chi^v_\sigma>0$  then $v$  is the $\omega$-limit  of an orbit in $\gamma$.
\end{itemize}
Let $\gamma$ be an edge with end-points $v$ and $v'$ and opposite facets
$\sigma$ and $\sigma'$, respectively. This means that
$(v,\gamma, \sigma)$ and $(v',\gamma, \sigma')$ are corners of $\Gamma_{\nund}$. If $X_{\nund,A}$ does not have singularities  in $\inter (\gamma)$, then $\inter (\gamma)$ consists of a single heteroclinic orbit with $\alpha$-limit $v$ and $\omega$-limit $v'$  if and only if $\chi^{v}_{\sigma}<0$ and $\chi^{v'}_{\sigma'}>0$.
This type of edges will be referred to as {\it flowing edges}.
The vertices $v=s(\gamma)$ and $v^\prime=t(\gamma)$ are respectively called the {\em source}  and {\em target}  of the flowing edge $\gamma$ and we will write $v \buildrel\gamma \over\longrightarrow v'$ to express it. When the two characters $\chi^{v}_{\sigma}=\chi^{v'}_{\sigma'}=0$ the edge $\gamma$ its called \emph{neutral}.

\end{remark}

\begin{defn}\label{def:regular}
A polymatrix replicator $X_{\nund,A}$ is called \emph{regular} if it is non-degenerate and moreover every edge is either neutral or a flowing edge.
\end{defn}

Given $v=v_{\j}$ with $\j=(j_1,\ldots, j_p)\in \Vset$  consider the vertex neighborhood 
\[N_v:=\{q\in \Gamma_\nund :\,\, 0\leq f_{k}(q)\leq 1, \; \forall k\in\Fset\backslash \{j_1,\ldots, j_p\} \; \} . \]
Rescaling the defining functions  $f_k$   we may assume these neighborhoods  are pairwise disjoint. See Remark~\ref{local coordinate system}.

 For any edge $\gamma$ with end-points $v$ and $v'$ we  define a tubular neighborhood  connecting $N_{v}$ to $N_{v'}$ by 
\[N_{\gamma}:=\{q\in\Gamma_\nund\backslash (N_{v}\cup N_{v'})\, \colon \, 0\leq f_k(q)\leq 1,\; \forall \, k\in\Fset \text{ with } \gamma\subset \sigma_{k} \}. \]
Again we may assume that these neighborhoods are pairwise disjoint between themselves. Finally we define the {\em edge skeleton's  tubular  neighborhood}
of $\Gamma_{\nund}$ to be
 \begin{equation}\label{eq:N-Gamma}
N_{\Gamma_\nund}:=(\cup_{v\in V}N_v)\cup(\cup_{\gamma\in E}N_\gamma) .
\end{equation} 

 The next step is to define the rescaling map $\Psi^{\nund,A}_\epsilon$ on  $N_{\Gamma_\nund}\backslash \partial \Gamma_{\nund}$. See ~\cite[Definition $5.2$]{ADP2020}. We will write
 $f_\sigma$ to denote the affine function  $f_k$ associated with the facet
 $\sigma=\sigma_k$ with $k\in\Fset$.

\begin{defn}\label{defn:vcoor.ch.}
Let  $\epsilon>0$ be a small parameter.
The $\epsilon$-rescaling  coordinate system $$\Psi_\epsilon^{\nund,A}:N_{\Gamma_\nund}\backslash \partial\Gamma_\nund\to\Rr^F$$  maps $q\in N_{\Gamma_\nund}$ to $y:=(y_{\sigma})_{\sigma\in F}$ where 
\begin{itemize}
\item[$\bullet$] if $q\in N_v$ for some vertex $v$:
$$y_{\sigma}=\left\{ \begin{array}{cll}
-\epsilon^2 \log  f_\sigma(q)  & \text{ if } & v\in \sigma  \\
0 &   \text{ if } &v\notin \sigma 
\end{array}
\right. $$
\item[$\bullet$] if $q\in N_\gamma$ for some edge $\gamma$:
$$y_{\sigma}=\left\{ \begin{array}{cll}
-\epsilon^2 \log  f_\sigma(q)  & \text{ if } &\gamma\subset \sigma \\
0 &   \text{ if } &\gamma\not\subset \sigma
\end{array}
\right.$$
\end{itemize}
\end{defn}

We  now turn to the space where these rescaling coordinates take values.
 For a given vertex $v\in V$ we define
\begin{equation}\label{eq:pi-v}
\Pi_v:=\{\, (y_\sigma)_{\sigma\in F}\in \Rr_+^{F}\,\colon\, y_\sigma=0, \quad \forall \sigma\notin F_v\,\} 
\end{equation}
where $\R_+=[0,+\infty)$.
Since $\{f_\sigma\}_{\sigma \in F_v}$ is a coordinate system over $N_v$ and the function  $h :(0,1]\to [0,+\infty)$,
$h(x):=-\log x$, is a diffeomorphism, the restriction   $\Psi_\epsilon^{\nund,A}: N_v\backslash \partial\Gamma_\nund\to \Pi_v$ is also a diffeomorphism  denoted by $\Psi_{\epsilon,v}^{\nund,A}$.

If $\gamma$ is an edge connecting two corners $(v,\sigma)$ and $(v', \sigma')$, $F_{v}\cap F_{v'}=\{\sigma \in F : \gamma\subset\sigma\}$ and we define
\begin{equation}
\label{eq:pi-gamma}
\Pi_\gamma :=\{\, (y_\sigma)_{\sigma\in F}\in \Rr_+^{F}\,\colon\, y_\sigma=0 \quad \mbox{when}\, \gamma\not\subset \sigma\,\}\;.
\end{equation}
Then $\Psi_\epsilon^{\nund,A}(N_\gamma\backslash \partial\Gamma_{\nund}) = \Pi_\gamma= \Pi_{v}\cap \Pi_{v'}$  
has dimension $d-1$ while $\Pi_v=\Psi_{\epsilon. v}^{\nund,A}(N_v\backslash \partial\Gamma_{\nund})$  has  dimension $d$. In particular the map $\Psi^{\nund,A}_{\epsilon,v}$ is not injective over $N_\gamma$.  
See Figure~\ref{edge-corners}.

\begin{center}
	\begin{figure}[h]
		\includegraphics[width=8cm]{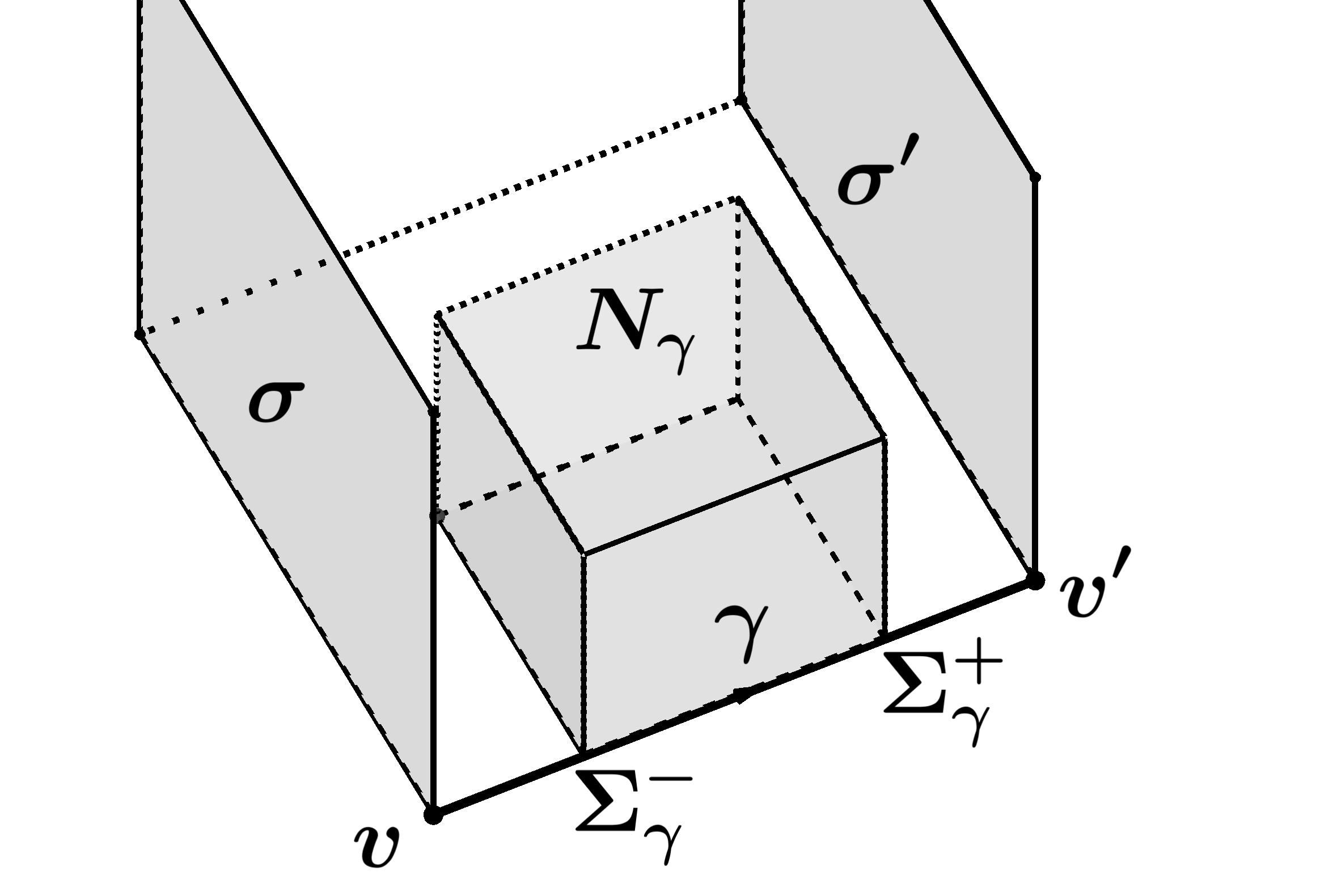}
		\caption{\footnotesize{An edge connecting two corners.}} \label{edge-corners}
	\end{figure}
\end{center}

 \begin{defn}\label{defn:dualcones}
The {\em dual cone} of   $\Gamma_\nund$ is defined to be
$$\CC^\ast(\Gamma_\nund) :=\bigcup_{v\in V}\Pi_v\;, $$
where $\Pi_v$ is the sector in~\eqref{eq:pi-v}. 
\end{defn}

Hence $\Psi_\epsilon^{\nund,A}:N_{\Gamma_{\nund}}\backslash \partial\Gamma_{\nund}\to \CC^\ast(\Gamma_\nund)$.

Denote by $\{ \varphi_{\nund,A}^t:\Gamma_{\nund}\to\Gamma_{\nund} \}_{t\in\R}$ the flow of the vector field $X_{\nund,A}$.
Given a flowing edge $\gamma$ with source $v=s(\gamma)$ and target $v'=t(\gamma)$ we introduce the cross-sections
$$\Sigma_{\gamma}^-:=(\Psi_{v,\epsilon}^{\nund,A})^{-1}(\inter(\Pi_{\gamma}))\quad\mbox{and}\quad \Sigma_{\gamma}^+:=(\Psi_{v',\epsilon}^{\nund,A})^{-1}(\inter(\Pi_{\gamma}))  $$
transversal to the flow   $\varphi^t_{\nund,A}$. The sets $\Sigma_{\gamma}^-$ and $\Sigma_{\gamma}^+$ are inner facets of the tubular neighborhoods   $N_v$ and $N_{v^\prime}$ respectively. 
Let $\Dscr_{\gamma}$ be the set of points $x\in \Sigma^-_\gamma$ such that the forward  orbit $\{\varphi_{\nund,A}^t(x)\colon t>0\}$ has a first transversal intersection with $\Sigma^+_\gamma$.
The global Poincar\'e  map 
$$P_{\gamma}: \Dscr_{\gamma}\subset
\Sigma_\gamma^-  \to \Sigma_\gamma^+ $$
is defined by
$P_{\gamma}(x):= \varphi_{\nund,A}^{\tau(x)}(x)$,
where 
$$ \tau(x)= \min\{\, t>0\,\colon\, 
\varphi_{\nund,A}^t(x) \in \Sigma_{\gamma}^+\, \} \;.$$

 To simplify some of the following convergence statements  we use the terminology in~\cite[Definition $5.5$]{ADP2020}.
 \begin{defn}\label{remark:convergence}
 	Suppose we are given a family of functions $F_\epsilon$ with varying domains $\UU_\epsilon$.
 	Let $F$ be another function with domain $\UU$. Assume that  all these functions have the same  
 	target and source spaces, which are assumed to be linear spaces. 
 	We will say that\, 
 	$ \lim_{\epsilon \rightarrow 0^+} F_\epsilon = F$\,
 	{\em in the  $C^k$ topology},
 	to mean that:
 	\begin{enumerate}
 		\item {\em domain convergence}: for  every compact subset $\, K\subseteq\UU$, 
 		we have $K \subseteq\UU_\epsilon$ for every small enough $\epsilon>0$, and 
 		\item {\em uniform convergence on compact sets}:
 		$$ \lim_{\epsilon \rightarrow 0^+}\;
 		\max_{0\leq i\leq k} \sup_{ u\in K }
 		\left|\, D^i \left[ 
 		F_\epsilon(u) - F(u)
 		\right]\,
 		\right|\; =\; 0\;.$$
 	\end{enumerate}
 	Convergence in the $C^\infty$ topology means
 	convergence in the $C^k$ topology for all $k\geq 1$.
 	If  $F_\epsilon$ is a composition of two or more mappings then
 	its domain should be understood as the composition domain.
 \end{defn}

Let now
\begin{equation}\label{eq:pivepsilon1}
\Pi_{\gamma} (\epsilon):=\{\,y\in\Pi_{\gamma}\,\colon\, y_{\sigma} \geq\epsilon \quad\text{whenever}\quad \gamma \subset \sigma\} ,
\end{equation}
and define
$$F_{\gamma}^\epsilon:=\Psi_{v',\epsilon}^{\nund,A}\circ P_{\gamma} \circ(\Psi_{v,\epsilon}^{\nund,A})^{-1}.$$ 
Notice that $\lim_{\epsilon\to 0} \Pi_{\gamma} (\epsilon)=\inter(\Pi_{\gamma})$.

\begin{lemma}
\label{pgamma}
For a given $k\geq 1$, there exists a number $r$ such that the following limit holds in the $C^k$ topology, 
\begin{align*}
\lim_{\epsilon\to0^+}F^\epsilon_{\gamma|_{\UU_{\gamma}^\epsilon}}=\id_{\Pi_{\gamma}},
\end{align*} 
where  $\UU_{\gamma}^\epsilon\subset \Pi_{\gamma}(\epsilon^r)$ is the domain of $F_{\gamma}^\epsilon$. 
\end{lemma}
\begin{proof}
See~\cite[Lemma 7.2]{ADP2020}.
\end{proof}

Hence, since  the  global Poincar\'{e} maps  converge towards the  identity map  as we approach  the heteroclinic orbit, the asymptotic behavior of the flow is solely determined by local Poincar\'{e} maps.  

From Definition~\ref{skeletoncharacter}, for any  vertex $v$, the vector $\chi^v$ is tangent to $\Pi_v$, in the sense that $\chi^v$ belongs to the linear span of the sector $\Pi_v$. 
 Let 
\begin{equation}\label{eq:pivepsilon}
\Pi_v(\epsilon):=\{\,y\in\Pi_v\,\colon\, y_\sigma \geq\epsilon \quad\text{for all}\quad \sigma\in F_v\, \}
\end{equation}
Using the notation of Definition~\ref{skeletoncharacter} we have
\begin{lemma}
\label{lemma:rescal}
We have 
\[(\Psi_{v,\epsilon}^{\nund,A})_\ast X_{\nund,A}=\epsilon^2 
\,\left(\tilde{X}_{v,\sigma}^\epsilon\right)_{\sigma\in F} ,\]
where $$\tilde{X}_{v,\sigma}^\epsilon(y):=
\left\{ \begin{array}{cll}
-H_\sigma\left((\Psi_{v,\epsilon}^{\nund,A})^{-1}(y) \right) &\text{if} & \sigma\in F_v\\
0  &\text{if} & \sigma\notin F_v\\
\end{array}  
\right. ,$$
Moreover, given $k\geq 1$ there exists $r>0$ such that
 the following limit holds in the $C^k$ topology
\[\lim_{\epsilon\to 0}\, ( \tilde{X}_v^\epsilon)_{|_{\Pi_v(\epsilon^r)}}= \chi^v\;.\]
\end{lemma}

\begin{proof}
See~\cite[Lemma $5.6$]{ADP2020}.
\end{proof}
Consider a vertex $v$ with an incoming flowing-edge $v_\ast\buildrel\gamma \over\longrightarrow v$ and an outgoing flowing-edge  $v \buildrel\gamma'\over\longrightarrow v'$. Denote by $\sigma_\ast$   the facet opposed to $\gamma^\prime$ at $v$. We define the sector  
\begin{equation}\label{eq:hyperplane}                
\Pi_{\gamma,\gamma^\prime} :=\left\{\, y\in {\rm int}(\Pi_{\gamma})\,\colon\,  y_\sigma-\frac{\chi^{v}_\sigma}{\chi^{v}_{\sigma_\ast}}\, y_{\sigma_\ast} > 0, \, \forall  \sigma \in F_{v},\; \sigma\ne \sigma_\ast  \; \right\} 
\end{equation}
and the linear map $L_{\gamma,\gamma^\prime}:\Pi_{\gamma,\gamma^\prime}\to\Pi_{\gamma^\prime}$ by
\begin{equation}\label{eq:L:gamma:gammaprime}
L_{\gamma,\gamma^\prime}(y):=  \left(\, y_\sigma-\frac{\chi^{v}_\sigma}{\chi^{v}_{\sigma_\ast}} \, y_{\sigma_\ast}\, \right)_{\sigma\in F}\; .
\end{equation}
 Notice that $\Pi_{\gamma^\prime}=\{y\in \Pi_{v}: y_{\sigma_\ast}=0\} $ as well as $\Pi_\gamma$ are facets to $\Pi_v$.

\begin{proposition}
	\label{proposition:skeletonflow}
The sector $\Pi_{\gamma,\gamma'}$ consists of all points $y\in {\rm int} (\Pi_{\gamma})$ which can be connected to some point
$y'\in {\rm int} (\Pi_{\gamma'})$ by a line segment inside the
ray $\{\, y+t\chi^v\,
\colon t\geq 0\, \}$. Moreover, if $y\in \Pi_{\gamma,\gamma'}$ then
the other endpoint  is $y'= L_{\gamma,\gamma'}(y)$.
\end{proposition}

\begin{proof}
See~\cite[Proposition $6.4$]{ADP2020}.
\end{proof}

Given flowing-edges  $\gamma$ and $\gamma'$ such that
$t(\gamma)=s(\gamma^\prime)=v$ we denote by  $\Dscr_{\gamma,\gamma^\prime}$  the set of points $x\in  \Sigma_{v,\gamma}$ 
such that the  forward orbit  $\{ \varphi_{\nund,A}^t(x)\colon t\geq 0 \}$  has a first transversal intersection with $\Sigma_{v,\gamma^\prime}$.  The  local Poincar\'e map 
$$P_{\gamma,\gamma^\prime}: \Dscr_{\gamma,\gamma^\prime}\subset
\Sigma_{\gamma}^+\to  \Sigma_{\gamma'}^-$$
is defined by
$P_{\gamma,\gamma^\prime}(x):= \varphi_{\nund,A}^{\tau(x)}(x)$, where
$$ \tau(x):= \min\{\, t>0\,\colon\, 
\varphi_{\nund,A}^{t}(x) \in \Sigma_{\gamma'}^-\, \} \;.$$

\begin{lemma}
	\label{vertexpoincare} 
Let \,$\UU_{\gamma,\gamma^\prime}^\epsilon\subset \Pi_{\gamma}(\epsilon^r)$  be the domain of the map $$F_{\gamma,\gamma^\prime}^\epsilon:=\Psi_{v,\epsilon}^{\nund,A}\circ P_{\gamma,\gamma^\prime} \circ(\Psi_{v,\epsilon}^{\nund,A})^{-1}.$$
Then for a given $k\geq 1$ there exist   $r>0$ such that 
\begin{align*}
\lim_{\epsilon\to0^+}\left(F_{\gamma,\gamma^\prime}^\epsilon\right)_{|_{\UU_{\gamma,\gamma^\prime}^\epsilon}}=
 L_{\gamma,\gamma^\prime} \; 
\end{align*}
in the $C^k$ topology.
\end{lemma}
\begin{proof}
See~\cite[Lemma $7.5$]{ADP2020}.
\end{proof}

Given  a chain of flowing-edges
\begin{equation*} 
v_0\buildrel\gamma_0 \over\longrightarrow  v_1 \buildrel\gamma_1 \over\longrightarrow  v_2 \longrightarrow  \ldots \longrightarrow  v_m \buildrel\gamma_m \over\longrightarrow v_{m+1}
\end{equation*}
the sequence $\xi =(\gamma_0,\gamma_1,\ldots, \gamma_m)$ is called a \emph{heteroclinic path}, or  a \emph{heteroclinic cycle} when $\gamma_{m}=\gamma_0$.

\begin{defn}\label{poincaremap}
 Given a  heteroclinic path  $\xi =(\gamma_0,\gamma_1,\ldots, \gamma_m)$:
 \begin{itemize}
 \item[1)] The {\em Poincar\'e map} of  a polymatrix replicator $X_{\nund,A}$ along $\xi$ is the composition
$$P_\xi:=( P _{\gamma_m}\circ P_{\gamma_{m-1},\gamma_m}) \circ\ldots\circ (P_{\gamma_1}\circ P_{\gamma_0,\gamma_1}),$$
whose domain is denoted by $\UU_\xi$. 
\item[2)] The \emph{skeleton flow map (of $\chi$) along  $\xi$} is the composition map $\pi_\xi:\Pi_\xi\to \Pi_{\gamma_m}$ defined by
$$\pi_\xi:=L_{\gamma_{m-1},\gamma_m}\circ\ldots\circ L_{\gamma_0,\gamma_1}\; ,$$ 
whose domain is   
\begin{align*}
\Pi_\xi & := {\rm int}(\Pi_{\gamma_0}) \cap 
\bigcap_{j=1}^{m} (L_{\gamma_\ast,\gamma_j}\circ\ldots\circ L_{\gamma_0,\gamma_1})^{-1}
 {\rm int}(\Pi_{\gamma_j})   \;. 
\end{align*}
\end{itemize}
\end{defn}

The previous lemmas~\ref{pgamma} and~\ref{vertexpoincare} imply that  given a  heteroclinic path  $\xi$, the  asymptotic behavior of the  Poincar\'e map $P_\xi$  along $\xi$
is given by the  Poincar\'e map $\pi_\xi$ of   $\chi$.
\begin{proposition}
\label{pathpoincare}
Let $\UU_\xi^\epsilon$  be the domain of the map
$$F_\xi^\epsilon:=\Psi_{v_m,\epsilon}^{\nund,A}\circ P_\xi \circ \left(\Psi_{v_0,\epsilon}^{\nund,A}\right)^{-1}$$
from $\Pi_{\gamma_0}(\epsilon^r)$
into $\Pi_{\gamma_m}(\epsilon^r)$.
Then 
\begin{align*}
\lim_{\epsilon\to0^+}\left(F_{\xi}^\epsilon\right)_{|_{\UU_{\xi}^\epsilon}}=\pi_\xi 
\end{align*}
in the $C^k$ topology.
\end{proposition}

\begin{proof}
See~\cite[Proposition $7.7$]{ADP2020}.
\end{proof}

To analyze the dynamics of the flow of the skeleton vector field $\chi$ we introduce the concept of {\em structural set} and  its associated skeleton flow  map. See~\cite[Definition $6.8$]{ADP2020}.

\begin{defn}\label{structural:set}
	A non-empty set of flowing-edges $S$ 
	is said to be a {\em structural set} for  $\chi$ if
	every heteroclinic cycle  contains an edge in $S$.
\end{defn}

Structural sets are in general not unique.
We say that a heteroclinic path $\xi=(\gamma_0,\ldots,\gamma_m)$ is an $S$-branch if
\begin{enumerate}
	\item $\gamma_0,\gamma_m\in S$,
	\item $\gamma_j\notin S$ for all $j=1,\ldots, m-1$.
\end{enumerate}
Denote by $\Bscr_S(\chi)$ the set of all $S$-branches.

\begin{defn}	\label{def skeleton flow map}
The \textit{skeleton flow  map}  $\pi_S:D_S\to\Pi_S$ is defined by  
$$
\pi_S(y):=\pi_\xi(y)\quad \text{ for all } \; y\in\Pi_\xi,
$$
where 
$$
D_S:=\cup_{\xi\in\Bscr_S(\chi)}\Pi_\xi	
\quad \text{ and } \quad 
\Pi_S:=\cup_{\gamma\in S}\Pi_\gamma.
$$
\end{defn}

%%%%%%%%%%%%%%%%

The reader should picture  $\pi_S:D_S\to\Pi_S$ as the first return map  of the piecewise linear flow 
of $\chi$ on $\CC^\ast(\Gamma_\nund)$ to the system
of cross-sections $\Pi_S$.  
The following, see~\cite[Proposition $6.10$]{ADP2020}, provides a sufficient condition
for  the skeleton flow map $\pi_S$ to be a closed dynamical system.

\begin{proposition}  \label{partition}
	Given a  skeleton vector field $\chi$ on $\CC^\ast(\Gamma_\nund)$  with a structural set $S$, assume  
	\begin{enumerate}
		\item every edge of $\Gamma_{\nund}$ is either neutral or a flowing-edge,
		\item every vertex $v$ is of saddle type, i.e., $\chi^v_{\sigma_1} \chi^v_{\sigma_2}<0$ for some facets $\sigma_1,\sigma_2\in F_v$.
	\end{enumerate}
	Then  
	$$ \hat D_S := \bigcap_{n\in\Zz} (\pi_S)^{-n}(D_S)$$
	is a Baire space with full Lebesgue measure in $\Pi_S$ and 
	   $\pi_S:\hat D_S\to \hat D_S$ is a homeomorphism.
\end{proposition}

Given a structural set $S$  any orbit of the flow  $\varphi_{\nund,A}^t$ that shadows some heteroclinic cycle must intersect the  cross-sections $\cup_{\gamma\in S} \Sigma^+_\gamma$  recurrently. The following map encapsulates the semi-global dynamics of these orbits.

\begin{defn}
	Given  $X_{\nund,A}$, let  $S$ be a structural set of its skeleton vector field. We define  $P _S:\UU_{S} \subset \Sigma_S\to \Sigma_S$ setting  $\Sigma_S:= \cup_{\gamma\in S} \Sigma_\gamma^+$,
	$\UU_{S}:= \cup_{\xi\in \Bscr_S(\chi)} \UU_{\xi}$
	and $P_S(p):=P_\xi(p)$ for all $p\in \UU_{\xi}$.
	The domain components $\UU_{\xi}$ and $\UU_{\xi'}$
	are disjoint for branches $\xi\neq \xi'$  in $\Bscr_S(\chi)$.
\end{defn}

Up to a time reparametrization, the  map
$P_{S}:D_{S} \subset \Sigma_S\to \Sigma_S$ embeds  in the flow $\varphi^t_{(\nund,A)}$.
In this sense the dynamics of   $P_{S}$ encapsulates the qualitative behavior
of the flow $\varphi_X^t$ of $X$ along the edges of $\Gamma_{\nund}$.

\begin{theorem}\label{asymp:main:theorem}
	Let $X_{\nund,A}$ be a regular polymatrix replicator 
	with skeleton vector field $\chi$. If $S$ is a structural set of $\chi$ then  
	$$ \lim_{\epsilon\to 0^ +}  \Psi_\epsilon \circ \Poin{X}{S}\circ
	(\Psi_\epsilon)^{-1} = \skPoin{\chi}{S}  $$
	in the $C^\infty$ topology, in the sense of Definition~\ref{remark:convergence}.
\end{theorem}

\begin{proof}
See~\cite[Theorem $7.9$]{ADP2020}.
\end{proof}

 %%%%%%%%%%%%%%%%%%%%%%%%%%%%%%%
\section{Hamiltonian character of the asymptotic dynamics}
\label{sec:hamiltonian-character}

In this section we  discuss the Poisson geometric properties of the Poincar\'e maps $\pi_\xi$ in the case of Hamiltonian polymatrix replicator equations. 
Given a generic Hamiltonian polymatrix replicator, $X_{\nund,A_0}$, we study its asymptotic Poincar\'e maps,
proving that they are Poisson maps.

Let $X_{\nund, A}$ be a conservative polymatrix replicator, $q$ a formal equilibrium,
$A_0$ and $D$ as in Definition~\ref{CPR}, and 
\begin{equation}\label{eq:hamiltonian-101}
h(x) =\sum_{\beta=1}^p\sum_{j\in [\beta]}\lambda_\beta q^{\beta}_j\log x_j^{\beta}
\end{equation}
 its Hamiltonian function as in Theorem~\ref{conservative:hamiltonian}. 
The Hamiltonian~\eqref{eq:hamiltonian-101} belongs to a class of prospective constants of motion for vector fields on polytopes discussed in~\cite[Section $8$]{ADP2020}. Since the polymatrix replicator is fixed   we drop superscript ``$\nund, A$'' and use  $\Psi_{v,\epsilon}$ for the rescaling coordinate systems defined in Definition~\ref{defn:vcoor.ch.}. The following proposition gives the asymptotic constant of motion, on the dual cone, associated ${\mbox{to}\,h.}$

\begin{proposition}\label{asymptotic-constants} Given $\eta:\CC^\ast(\Gamma_\nund)\to \Rr$ defined by
\begin{equation}\label{eq:eta}
\eta(y):=\sum_{\beta=1}^p\sum_{j\in [\beta]} \lambda_\beta q^{\beta}_j y_j^{\beta},
\end{equation}
\begin{enumerate}
	\item 
	$\displaystyle \eta =\lim_{\epsilon \rightarrow 0^+} \epsilon^2  
	h\circ \left(\Psi_{v,\epsilon}\right)^{-1}$  over $\inter(\Pi_v)$ for any vertex $v$, with convergence in the $C^\infty$ topology;
	\item 
	$\displaystyle d\eta = \lim_{\epsilon \rightarrow 0^+} \epsilon^2
	\left[ \left(\Psi_{v,\epsilon}\right)^{-1} \right]^\ast \left( d h \right)$ over  $\inter(\Pi_v)$ for any vertex $v$,	with convergence in the $C^\infty$ topology;
	\item Since $h$ is invariant under the flow of ${X_{\nund,A}}$, i.e., $dh(X_{\nund,A})\equiv 0$,
	the function $\eta$ is invariant under the skeleton  flow of $\chi$, i.e., $d\eta(\chi)\equiv 0$.
\end{enumerate} 
\end{proposition}

\begin{proof}
See~\cite[Proposition $8.2$]{ADP2020}.
\end{proof}

We will use the following family of coordinate charts for the Poisson manifold $(\inter(\Gamma_\nund),\pi_{A_0})$ where $\pi_{A_0}$ is defined in~\eqref{eq:Poissonpmg}.

\begin{defn}\label{familly-of-charts}	Given a vertex $v=(e_{j_1},...,e_{j_p})$ of $\Gamma_\nund$,	we set $\hat{x}_\alpha:=(x_{k}^\alpha)_{k\in[\alpha]\setminus\{j_\alpha\} }$ and    $\hat{x}:=(\hat{x}^\alpha)_{\alpha}$, and define the projection map 
  $$P_v:\inter(N_v)\to(\mathbb{R}^{n_1-1}\times\ldots\times\mathbb{R}^{n_p-1}), \quad P_v(x):=\hat x .$$ 
$P_v$ is a diffeomorphism onto its image $(0,1)^{n-p}$ and
the inverse map  $\psi_v:=P^{-1}_v$  can be regarded as a local chart for the manifold $\inter(\Gamma_\nund)$.
\end{defn}

\begin{remark}
The projection map $P_v$ extends linearly to $\R^n$ and it is represented by the $(n-p)\times n$ block diagonal matrix 
 $$P_v = \diag(P_v^1,\ldots,P_v^p),$$
where  $P_v^\alpha$, for $\alpha=1,\ldots,p$, is the ${(n_\alpha-1)\times n_\alpha}$ constant matrix  obtained  from the identity matrix by removing its row $j_\alpha$.
\end{remark}

%As in Section~\ref{sec:polymatrix-replicators},
%given $x\in \Rr^n$, we denote by $D_x$ the $n\times n$ diagonal matrix
%$D_x := \diag(x_1,\dots ,x_n)$, and for each $\alpha\in\{1,\ldots, p\}$ we define the $n_\alpha\times n_\alpha$ matrix
%$$ T^\alpha_x:= x^\alpha\, \unit^t -I\;,$$
%and $T_x$ the $n\times n$ block diagonal matrix
%$T_x := \diag(T^1_x,\dots ,T^p_x)$.

Using the definitions of $D_x$ and $T_x$ given in Section~\ref{sec:polymatrix-replicators} we can state the following lemma.

\begin{lemma}
Consider the Poisson manifold $(\inter(\Gamma_\nund),\pi_{A_0})$ where $\pi_{A_0}$ is defined in~\eqref{eq:Poissonpmg}. Then for any vertex $v$, the  matrix representative of  $\pi_{A_0}$ in the local chart $\psi_v$ is
\begin{equation}\label{eq:representing-matrix-nv}
\pi^{\sharp_v}_{A_0}(\hat{x})=(-1)\, P_v\,T_x\, D_x \, A_0\, D_x\,T_x^t\,P_v^t.
\end{equation}
\end{lemma}

\begin{proof}
Notice that $\pi^{\sharp_v}_{A_0}(\hat{x}):=[\{x^\alpha_{k},x^\beta_{l}\}]$ with $\alpha,\beta=1,\ldots,p$ and $k\in[\alpha]\setminus\{j_\alpha\} $, $l\in [\beta]\setminus\{j_\beta\} $.
\end{proof}

We used the notation $\sharp_v$ instead of $\sharp$ to make it clear that the representing matrix is with respect to the local chart $\psi_v$.
The following trivial lemma gives us  the differential of the $\epsilon$-rescaling map $\Psi_{v,\epsilon}$  (in Definition~\ref{defn:vcoor.ch.}) for the coordinate chart $\psi_v$.
Given a vertex $v=(e_{j_1},\ldots, e_{j_p})$ and using the notation introduced in Definition~\ref{familly-of-charts} we write 
$D_{\hat x_\alpha}=\diag(x_1^\alpha,\ldots,x_{j_\alpha-1}^\alpha,x_{j_\alpha+1}^\alpha,\ldots,x_{n_\alpha}^\alpha)$
and denote by 
$D_{\hat x}$ the diagonal matrix $\diag(D_{\hat x_1}, \ldots, D_{\hat x_p})$. 

\begin{lemma}
The differential of the diffeomorphism  $${\Psi_{v,\epsilon}\circ\psi_v:P_v(\inter(N_v))\to \inter(\Pi_v)}$$ 
is given by
\[d_{\hat{x}}(\Psi_{v,\epsilon}\circ\psi_v)=-\epsilon^2
\, D_{\hat x}^{-1} .
%\diag(D^1(\Psi_{v,\epsilon}\circ\phi_v)_{\hat{x}^1},\ldots,D^p(\Psi_{v,\epsilon}\circ\phi_v)_{\hat{x}^p}),
\]
\end{lemma}

We push forward, by the  diffeomorphism $\Psi_{v,\epsilon}\circ\psi_v$, the Poisson structure $\pi^{\sharp_v}_{A_0}$ defined on   $P_v(\inter(N_v))$ to $\inter(\Pi_v)$.
The following lemma provides the matrix representative of the push forwarded Poisson structure.
In order to simplify the notation we set
\begin{equation}\label{notation-sec-5-1}
\mathbb{J}(\hat{x}):=-\epsilon^2
\, D_{\hat x}^{-1} \, P_v\,T_x\, D_x \,
\end{equation}
and for every $\alpha=1,\ldots,p$
\begin{equation}\label{notation-sec-5-2}
\mathbb{J}_\alpha(\hat{x}^\alpha):= -\epsilon^2
\, D_{\hat x_\alpha}^{-1}
\, P^\alpha_v\,T^\alpha_xD_{x^\alpha}.
\end{equation}

Notice that $\mathbb{J}(\hat{x})=\diag(\mathbb{J}_1(\hat{x}^1),\ldots,\mathbb{J}_p(\hat{x}^p))$.

\begin{lemma}
The diffeomorphism $\Psi_{v,\epsilon}\circ\psi_v$ pushes forward the Poisson structure $\pi^{\sharp_v}_{A_0}$ to the  Poisson structure $\pi^{\sharp_v}_{A_0,\epsilon}$ on $\inter(\Pi_v)$ where
\begin{align}\label{pi-hat}
\pi^{\sharp_v}_{A_0,\epsilon}(y)=(-1)(\mathbb{J} A_0\mathbb{J}^t)\circ(\Psi_{v,\epsilon}\circ\psi_v)^{-1}(y).
\end{align} 
\end{lemma}

\begin{proof}
See Definition~\ref{def:Poissonmap} and Remark~\ref{remark:push-pull-Poisson}. 
\end{proof}

%If all the faces $\sigma\in F_v$ have order one, 
The Poisson structure $\pi^{\sharp_v}_{A_0,\epsilon}$ is asymptotically equivalent to a linear Poisson structure. 
Let   
\begin{equation}\label{eq:Ev}
E_v=\mathrm{diag}(E^1_v,\ldots,E^p_v),
\end{equation}
 be the  $(n-p)\times n$ matrix  defined by diagonal
 blocks $E_v^\alpha$, for $\alpha=1,..,p$, where the $\alpha^{\rm th}$ block is  the  $(n_\alpha-1)\times n_\alpha$ matrix  in which  the column $j_\alpha$ is equal to $\mathbbm{1}_{n_\alpha-1}$ and every other column  $k_\alpha\neq j_\alpha$ is equal to $-e_{k_\alpha}\in\Rr^{n_\alpha-1}$.

\begin{lemma}\label{rescale:Poisson-from}
Given a vertex $v=(e_{j_1},...,e_{j_p})$, if $E_v$ is the matrix in~\eqref{eq:Ev} and $B_v:=E_vA_0E^t_v$, then  
\[\lim_{\epsilon\to 0^+} \frac{-1}{\epsilon^2}\mathbb{J}\circ (\Psi_{v,\epsilon}\circ\psi_v)^{-1}(y)=E_v,\]
over $\inter(\Pi_v)$ with convergence in $C^\infty$ topology. 
Consequently, 
\[\lim_{\epsilon\to 0^+} \frac{1}{\epsilon^4}\pi^{\sharp_v}_{A_0,\epsilon}(y)=B_v,\]
 over $\inter(\Pi_v)$ with convergence in $C^\infty$ topology. 
\end{lemma}

\begin{proof}
A simple calculation shows that  for every $\alpha=1,..,p$ 
\begin{align*}
\frac{-1}{\epsilon^2}\mathbb{J}_\alpha= %D^\alpha_{x,\nu}
%\scalemath{.8}{
\begin{pmatrix}
(x_1^\alpha-1)&\ldots&x_{j_\alpha-1}^\alpha&x_{j_\alpha}^\alpha&x_{j_\alpha-1}^\alpha&\ldots&x_{n_\alpha}^\alpha \\
&&&\vdots&&\\
x_1^\alpha&\ldots&(x_{j_\alpha-1}^\alpha-1)&x_{j_\alpha}^\alpha&x_{j_\alpha+1}^\alpha&\ldots&x_{n_\alpha}^\alpha\\
x_1^\alpha,&\ldots&x_{j_\alpha-1}^\alpha&x_{j_\alpha}^\alpha&(x_{j_\alpha+1}^\alpha-1)&\ldots&x_{n_\alpha}^\alpha\\
&&&\vdots&&&\\
x_1^\alpha&\ldots,&x_{j_\alpha-1}^\alpha&x_{j_\alpha}^\alpha&x_{j_\alpha+1}^\alpha&\ldots&(x_{n_\alpha}^\alpha-1)
  \end{pmatrix} .
\end{align*}
%where 
%$$
%D^\alpha_{x,\nu}:=\diag((x^\alpha_1)^{1-\nu_{1_\alpha}},\ldots,(x^\alpha_{j_\alpha-1})^{1-\nu_{j_\alpha-1}},(x^\alpha_{j_\alpha+1})^{1-\nu_{j_\alpha+1}},\ldots,(x^\alpha_{n_\alpha})^{1-\nu_{n_\alpha}})
%$$
For every $\sigma\in F_v$ and  $k\in[\alpha]\setminus\{j_\alpha\} $ we have 
\[\lim_{\epsilon\to 0^+} x_{k}^\alpha\circ(\Psi_{v,\epsilon}\circ\psi_v)^{-1}(y)=\lim_{\epsilon\to 0^+} e^{-\frac{y^\alpha_{k}}{\epsilon^2}}=0.\]
Considering that $\displaystyle x_{j_\alpha}^\alpha=1-\sum_{k\in[\alpha]\setminus\{j_\alpha\}}x_{k}^\alpha$, we get the first claim of the lemma and the second claim is an immediate consequence.
\end{proof}

Figure~\ref{figure:poisson-on-dual-cone} illustrates the case $\Gamma=\Delta^2$.

\begin{center}
\tikzset{->-/.style={decoration={
  markings,
  mark=at position #1 with {\arrow{>}}},postaction={decorate}}}
  
  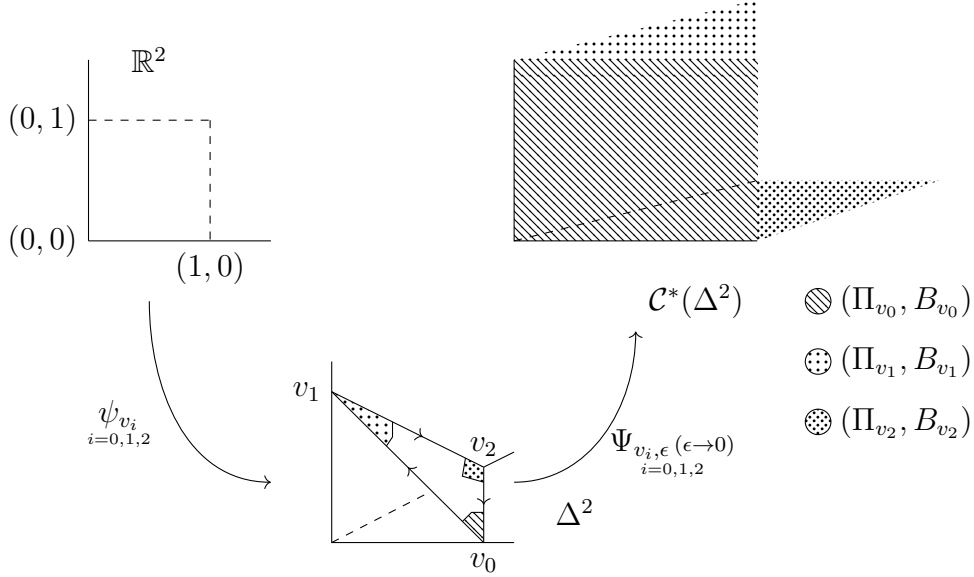
\begin{figure}
\begin{tikzpicture}[scale=.8]
\draw (0,0) node[below,left] {$(0,0)$}-- (3,0) ;\draw (0,0)--(0,3);\draw[dashed] (0,2)node[left]{$(0,1)$}--(2,2);\draw[dashed] (2,2)--(2,0) node[below]{$(1,0)$};
\node at (1,3) {$\mathbb{R}^2$};

%%%%%%%%%%%%%%

\draw (7,3)--(7,0)--(11,0);\draw[dashed] (7,0)--(11,1);
\fill[pattern=dots] (7,3)--(11,3)--(11,4)--(7,3);
\fill[pattern=crosshatch dots] (11,0)--(14,1)--(11,1)--(11,0);
\fill[pattern=north west lines] (7,0)--(7,3)--(11,3)--(11,0)--(7,0);
\node at (10,-1) {$\CC^\ast(\Delta^2)$};
\filldraw[pattern=north west lines] (12,-1) circle (6pt) node[right] {$\,\,(\Pi_{v_0},B_{v_0})$};
\filldraw[pattern=dots] (12,-2) circle (6pt) node[right] {$\,\,(\Pi_{v_1},B_{v_1})$};
\filldraw[pattern=crosshatch dots] (12,-3) circle (6pt) node[right] {$\,\,(\Pi_{v_2},B_{v_2})$};

%%%%%%%%%%%%%%%%%%%%%%%%%%

\draw (4,-2)--(4,-5)--(7,-5);\node at (8,-4.5) {$\Delta^2$};
\draw[dashed] (4,-5)--(5.6,-4.17);\draw (6.5,-3.75)--(7,-3.5);
\draw[->-=.5] (6.5,-5) node[below] {$v_0$} -- (4,-2.5) node[left] {$v_1$};\draw[->-=.5] (6.5,-3.75)--(6.5,-5);\draw[->-=.6] (4,-2.5)--(6.5,-3.75) node[above] {$v_2$};

\draw (5, -3)--(5,-3.25)--(4.9,-3.4);
\fill[pattern=dots] (4,-2.5)--(5, -3)--(5,-3.25)--(4.9,-3.4)--(4,-2.5);

\draw (6.15,-4.65)--(6.3,-4.5)--(6.5,-4.5);
\fill[pattern=north west lines] (6.5,-5)--(6.15,-4.65)--(6.3,-4.5)--(6.5,-4.5)--(6.5,-5);

\draw (6.2,-3.6)--(6.15,-3.9)--(6.5,-4);
\fill[pattern=crosshatch dots] (6.5,-3.75)--(6.2,-3.6)--(6.15,-3.9)--(6.5,-4)--(6.5,-3.75);

%%%%%%%%%%%%%%%%%%%%%%%%%

\draw[->] (1,-1) to[out=270,in=180] (3,-4);
\draw[->] (7,-4) to[out=0,in=270] (9,-1.5);
\node at (.5,-3) {$\overset{\displaystyle\psi_{v_i}}{\scriptscriptstyle i=0,1,2}$};
\node at (9.6,-3.5)  {$\overset{{\displaystyle\Psi_{v_i,\epsilon}}\,(\epsilon\to 0)}{\scriptscriptstyle i=0,1,2}$};
\end{tikzpicture}
\caption{\footnotesize{Poisson structures on the dual cone.}}
\label{figure:poisson-on-dual-cone}
\end{figure}
\end{center}

\begin{remark}  
The same linear Poisson structure $B_v:=E_vA_0E^t_v$ appears in~\cite[Theorem $3.5$]{AD2015}.
\end{remark}

\begin{lemma}\label{lemma:assymp-Hamil}
For a given vertex $v=(e_{j_1},...,e_{j_p})$, let $\chi^v$ be the skeleton character of $X_{\nund,A}$, as in Definition~\ref{skeletoncharacter}.
Then \[\chi^v=B_v d\eta_v, \]
where $\eta_v$ is the restriction of the function $\eta$ (defined in~\eqref{eq:eta}) to $\inter(\Pi_v)$. In other words,
$\chi^v$ restricted to $\inter(\Pi_v)$ is Hamiltonian with respect to the constant Poisson structure $B_v$ having $\eta_v$ as a Hamiltonian function.
\end{lemma}

\begin{proof}
We use the notation $X^v_{(\nund,A_0)}(\hat{x}):=(d_xP_v)X_{\nund,A}(x)$ for the local expression of the replicator vector field $X_{\nund,A}$ in the local chart $\psi_v$. 
If we write the function $h(x)$, defined in~\eqref{eq:hamiltonian-101}, as $h(x)=h(\psi_v\circ P_v(x))$ then
\[d_xh=(P_v)^td_{\hat{x}}(h\circ\psi_v)(\hat{x}).\]
Notice that $dP_v=P_v$.  By Theorem~\ref{conservative:hamiltonian}, $X_{\nund,A_0}=\pi_{A_0}d h$. Locally, 
\begin{align*}
X^v_{(\nund,A_0)}(\hat{x})=P_vX_{\nund,A}(x)=P_v\pi_{A_0}P^t_vd_{\hat{x}}(h\circ \psi_v).
\end{align*}
Similarly, writing $h\circ\psi_v(\hat{x})=h\circ\psi_v\circ(\Psi_{v,\epsilon}\circ\psi_v)^{-1}\circ(\Psi_{v,\epsilon}\circ\psi_v)(\hat{x})$ we have
\[d_{\hat{x}}(h\circ\psi_v)=(d_{\hat{x}}(\Psi_{v,\epsilon}\circ\psi_n))^t d_y(h\circ(\Psi_{v,\epsilon})^{-1}).\]
The vector field $\tilde{X}_v^\epsilon$ defined in Lemma~\ref{lemma:rescal} is
\begin{align*}
\tilde{X}_v^\epsilon&=\frac{1}{\epsilon^2}(d_x(\Psi_{v,\epsilon})X_{\nund,A})
=\frac{1}{\epsilon^2}(d_{\hat{x}}(\Psi_{v,\epsilon}\circ\psi_v)X^v_{(\nund,A_0)})\\
&=\frac{1}{\epsilon^2}(d_{\hat{x}}(\Psi_{v,\epsilon}\circ\psi_v)P_v\pi_{A_0}P^t_vd_{\hat{x}}(\Psi_{v,\epsilon}\circ\psi_n))^t d_y(h\circ(\Psi_{v,\epsilon})^{-1})\\
&=\frac{1}{\epsilon^4}\pi^{\sharp_v}_{A_0,\epsilon} \boldsymbol{\Big (}\epsilon^2((\Psi_{v,\epsilon})^{-1})^\ast d_x h \boldsymbol{\Big)},
\end{align*}
where in the second equality we use $\psi_v\circ P_v=\mathrm{Id}$. 
Then, applying Lemma~\ref{lemma:rescal}, Lemma~\ref{rescale:Poisson-from}, and Proposition~\ref{asymptotic-constants}, the result follows. 
Notice that $\Pi_v(\epsilon^r)\subset\inter(\Pi_v)$. 
 \end{proof}

Our aim is to show that for a given  heteroclinic path  $\xi =(\gamma_0,\gamma_1,\ldots, \gamma_m)$, the skeleton flow map of $\chi$ along  $\xi$ (see Definition~\ref{poincaremap}),
$$\pi_\xi:=L_{\gamma_{m-1},\gamma_m}\circ\ldots\circ L_{\gamma_0,\gamma_1}\; ,$$ 
restricted to the level set of $\eta$, is a Poisson map.
Notice that the Poisson structure $B_v$ is only defined in $\inter(\Pi_v)$ and neither $\Pi_\gamma$ nor $\Pi_{\gamma^\prime}$ are submanifolds of $\inter(\Pi_v)$.  So we need to define Poisson structures on the sections $\Pi_{\gamma_i,\gamma_{i+1}}$ for all $i=0,...,m$.  

For the heteroclinic path
\begin{equation}\label{eq:path0}
\xi: v_0\buildrel\gamma_0 \over\longrightarrow  v_1 \buildrel\gamma_1 \over\longrightarrow  v_2 \longrightarrow  \ldots \longrightarrow  v_m \buildrel\gamma_m \over\longrightarrow v_{m+1}, 
\end{equation}
we store in the $i^{th}$ column of the matrix  
$$J_\xi=\begin{bmatrix}
j_{01}&j_{11}&\ldots&j_{(m+1)1}\\
j_{02}&j_{12}&\ldots&j_{(m+1)2}\\
\vdots&\vdots&\vdots&\vdots\\
j_{0p}&j_{1p}&\ldots&j_{(m+1)p}\\
\end{bmatrix},$$
 the indices of the non zero components of the vertex $v_i=(e_{j_{i1}},e_{j_{i2}},\ldots,e_{j_{ip}})$.  
%We start with a vertex ${v_1}=(e_{j_1},...,e_{j_p})$ with an incoming flowing-edge $v_0\buildrel\gamma \over\longrightarrow {v_1}$ and an outgoing flowing-edge  ${v_1} \buildrel\gamma^\prime\over\longrightarrow v_2$. 
By construction of $\Gamma_\nund$,  there exists $(\xi_0,\xi_1,\ldots,\xi_m)\in\{1,2,\ldots,p\}^{m+1}$
%, $\hat{k}_0\in[\alpha_0]\setminus\{j_{\alpha_0}\} $ and $\hat{k}_2\in[\alpha_2]\setminus\{j_{\alpha_2}\} $ 
such that $j_{(i-1)l}=j_{il}$ for $l\ne \xi_{i-1}$ and $j_{(i-1)\xi_{i-1}}\ne j_{i\xi_{i-1}}$, i.e. $\xi_{i-1}$ is the group containing the nonzero component that differ between the end points of the edge $\gamma_{i-1}$.  In order to simplify notations for every vertex $v_i$ in $\xi$ we denote $$r_i={j_{(i-1)\xi_{(i-1)}}}\quad\mbox{and}\quad s_i=j_{(i+1)\xi_i}.$$
First we consider the vertex $v_i$ with incoming and outgoing edges
\begin{align*}
&\gamma_{i-1}: v_{i-1}+t(0,\ldots,e_{s_{i-1}}-e_{r_i},\ldots,0),\\
&\gamma_i: v_i+t(0,\ldots,e_{j_{s_i}}-e_{r_i},\ldots,0),
\end{align*}
respectively, i.e. $\buildrel\gamma_{i-1} \over\longrightarrow  v_i \buildrel\gamma_{i} \over\longrightarrow$, where $t\in [0,1]$. 
 Notice that 
 \begin{align*}
%\scalemath{.85}{
&\Pi_{v_{(i-1)}}=\{y\in\mathbb{R}^n_+\,|\,y_{j_{(i-1)l}}=0, \; l=1,\ldots,p\},\\
&\Pi_{v_i}=\{y\in\mathbb{R}^n_+\,|\,y_{j_{il}}=0,\; l=1,\ldots,p\}.
\end{align*}
Since $j_{(i-1)l}=j_{il}$ for $l\ne \xi_{i-1}$ and $\Pi_{\gamma_{i-1}}=\Pi_{v_{(i-1)}}\cap \Pi_{v_i}$ we have
\[
\Pi_{\gamma_{i-1}}=\{y\in\mathbb{R}^n_+\,|\,y_{j_{s_{i-1}}}=y_{j_{(i-1)l}}=0, \; l=1,\ldots,p\}.
\]

 The opposite facet to $\gamma_i$ at ${v_i}$ is then $\sigma_\ast:=\{y_{s_i}=0\}$ where we omitted the superscript $\xi_i$ from $y$ since it is evident that $j_{s_i}\in\xi_i$. We keep omitting the superscript  whenever there is no ambiguity. The sector defined in~\eqref{eq:hyperplane} is

\begin{small}
\begin{equation}\label{eq:hyperplane-replicator}
\Pi_{\gamma_{i-1},\gamma_i} :=\left\{\, y\in {\rm int}(\Pi_{\gamma_{i-1}})\,\colon\,  y_\sigma-\frac{\chi^{v_i}_\sigma}{\chi^{v_i}_{s_i}}\, y_{s_i} > 0, \, \forall  \sigma\ne j_{i1},\ldots,j_{i(m+1)}, s_i  \; \right\}.
%\Pi_{\gamma_{i-1},\gamma_i} :=\left\{\, y\in {\rm int}(\Pi_{\gamma_{i-1}})\,\colon\,
%  y^\alpha_{\hat{i}_\alpha}-\frac{(\chi^{{v_1}})^\alpha_{\hat{i}_\xi}}{(\chi^{{v_1}})_{j_{(i+1)\xi_i}}} \, y_{j_{(i+1)\alpha_i}}\,> 0, \, \forall \hat{i}_\alpha\ne j_{(i+1)\alpha_i}  \; \right\} .
\end{equation}
\end{small}

The skeleton flow map of $\chi$ at vertex ${v_i}$ is the linear map ${L_{\gamma_{i-1},\gamma_i}:\Pi_{\gamma_{i-1},\gamma_i}\to\Pi_{\gamma_i}}$ defined by
\begin{equation}\label{eq:L:gamma:gammaprime-replicator}
L_{\gamma_{i-1},\gamma_i}(y):=  \left(\, y_\sigma-\frac{\chi^{{v_i}}_\sigma}{\chi^{v_i}_{s_i}} \, y_{s_i}\, \right)_{\sigma}\; ,
\end{equation}
Notice that $L_{\gamma_{i-1},\gamma_i}(y)=\phi_{\chi^{v_i}}(\tau(y),y)$ where $\phi_{\chi^{v_i}}(\tau,y)=y+\tau \chi^{v_i},$ is the flow of the skeleton vector field $\chi^{v_i}$ and  $\tau(y):=-\frac{y_{s_i}}{\chi^{v_i}_{s_i}}$.  We denote $L^t_{\gamma_{i-1},\gamma_i}(y):=\phi_{\chi^{v_i}}(t\tau(y),y)$ where $t\in (0,1)$. More precisely
\[L^t_{\gamma_{i-1},\gamma_i}(y):=  \left(\, y_\sigma-t\frac{\chi^{{v_i}}_\sigma}{\chi^{v_i}_{s_i}} \, y_{s_i}\, \right)_{\sigma}\; .\]

 \begin{defn} \label{defn:tubular-neighborhood}
 We define by 
\begin{equation}
T_{\gamma_{i-1},\gamma_i}:=\bigcup\limits_{0<t<1} L_{\gamma_{i-1},\gamma_i}^t (\Pi_{\gamma_{i-1},\gamma_i}),
\end{equation}
the convex cone  containing the line segments of the flow of $\chi^{v_i}$ connecting the points in the domain of  $L_{\gamma_{i-1},\gamma_i}$  to their images. 
\end{defn}

We consider two cosymplectic foliations interior to each sector $\Pi_{v_i}$ in order to use the techniques introduced in Section~\ref{sec:poisson-poincare}. In the following lemma, we describe the Poisson structures on $\Pi_{\gamma_{i-1},\gamma_i}$ and $L_{\gamma_{i-1},\gamma_i}(\Pi_{\gamma_{i-1},\gamma_i})$.

\begin{lemma}\label{cosymplectic-foliations}
With the notation adopted in Lemma~\ref{lemma:assymp-Hamil}, let $\eta_{v_i}$ be the restriction of function $\eta$, defined in~\eqref{eq:eta},  to $\inter(\Pi_{v_i})$. Consider two functions  $G^{v_i}_r,\,G^{v_i}_s:T_{\gamma,\gamma^\prime}\to\mathbb{R}$ defined by
$G^{v_i}_r(y)=y_{r_i}$ and $G^{v_i}_s(y)=y_{s_i}$  then:
\begin{itemize}
\item[1)] Level sets of $(\eta_{v_i},G^{v_i}_r), (\eta_{v_i},G^{v_i}_s):T_{\gamma_{i-1},\gamma_i}\to\mathbb{R}^2$ partition  $T_{\gamma_{i-1},\gamma_i}$ into a  cosymplectic foliation $\mathcal{F}^{v_i}_r$ and  $\mathcal{F}^{v_i}_s$,  i.e. every leaf of these foliations is a cosymplectic submanifold of $(T_{\gamma_{i-1},\gamma_i},B_{v_i})$. Furthermore, every leaf $\Sigma$ of these foliations is a level transversal section to $\chi^{v_i}$ at every point $x\in\Sigma$;
\item[2)] Given two leafs\footnote{Notice that the flow of $\chi_{v_1}=X_{\eta_{v_1}}$ preserves $\eta_{v_1}$.} $\Sigma^l_r=(\eta_{v_i},G^{v_i}_r)^{-1}(c,d_l),\,l=1,2$, of $\mathcal{F}^{v_i}_r$ and  two leafs $\Sigma^l_s=(\eta_{v_i},G^{v_i}_s)^{-1}(c,d^\prime_l),\,l=1,2$, of $\mathcal{F}^{v_i}_s$, then the Poincar\'e map between any pair of these four leafs is a Poisson map.  
\end{itemize}
\end{lemma}

\begin{proof}
%For $r=0,2$, 
Clearly, 
\[\{\eta_{v_i},G_r^{v_i}\}=X_{\eta_{v_i}}(d G_r^{v_i})=\chi^{v_i}(d y_{r_i})=\chi^{v_i}_{r_i}, \]
and similarly $\{\eta_{v_i},G_s^{v_i}\}=\chi^{v_i}_{s_i}$. 
As before, using the notation $\sigma_{j_{ir}}$ for the facet $\{y_{j_{ir}}=0\}$, we see that ${\gamma_{i-1}}$  is a flowing edge from the corner $(v_{i-1}, \sigma_{i\xi_{i-1}})$ to the corner  $(v_i, \sigma_{s_i})$. So $\chi^{v_i}_{s_i}>0$. In a similar way we have  $\chi^{v_i}_{r_i}<0$. What we actually need is both of them to be nonzero. Then, both $\{\eta_{v_i},G^{v_i}_s\}$ and $\{\eta_{v_i}, G^{v_i}_r\}$ are second class constraints and consequently, their level sets are cosymplectic submanifolds (see Definition~\ref{cosymplectic}). The fact that $\Sigma$ is a level transversal section is clear.  

The Poincar\'e map  between $\Sigma^1_s,\Sigma^2_s$   is the translation
\begin{equation}\label{translations} \displaystyle P(y)=\phi_{\chi^{v_i}}\left( \frac{d_2-d_1}{\chi^{v_i}_{s_i}},y\right)=\left(\frac{d_2-d_1}{\chi^{v_i}_{s_i}}\right)\,\chi_{v_i}+y,
\end{equation}
and a similar translation for  $\Sigma^1_r,\Sigma^2_r$.  Clearly, these translations are Poisson maps. 

The Poincar\'e map between two level sets $\Sigma_r^l$ and $\Sigma_s^{l^\prime}$ is
\begin{equation}\label{poincare-map-0-2}
\displaystyle P(y)=\phi_{\chi^{v_i}}\left( \frac{d^\prime_{l^\prime}-y_{s_i}}{\chi^{v_i}_{s_i}},y\right)
\end{equation}
By Proposition~\ref{pois-poin-map}  this map is a Poisson map as well.
\end{proof}
\begin{remark}
Note that $y_{s_i}$ is not constant on $\Sigma_r^l$, so the map \eqref{poincare-map-0-2} is not a fixed time map of the flow $\phi_{\chi^{v_i}}$. Therefore, being Poisson is not a direct consequence of the flow being Hamiltonian. Furthermore, proving that this map is Poisson by direct calculation is not straightforward. This makes the contents of Section~\ref{sec:poisson-poincare} inevitable.
\end{remark}

Since the Poincar\'e maps can be considered between level sets of the functions $G_r^{v_i}$ and $G_s^{v_i}$, we state the following definition.
 
\begin{defn}
Let $\mathcal{\tilde{F}}^{v_i}_r$ and $\mathcal{\tilde{F}}^{v_i}_s$ be the foliations constituted by the level sets of the functions $G^{v_i}_r$ and $G^{v_i}_s$, respectively.
\end{defn}
 Every leaf of $\mathcal{\tilde{F}}^{v_i}_\ast,\, \ast=r,s$ is equipped with a Poisson structure, $\pi^{v_i}_\ast,\,\ast=r,s$ which has $\eta_{v_i}$ as a Casimir, and the level sets of this Casimir are the leafs of the cosymplectic foliation $\mathcal{F}^{v_i}_\ast$. The leafs of $\mathcal{\tilde{F}}^{v_i}_\ast$ can be identified (as Poison manifolds) through translations of type~\eqref{translations}.

\begin{defn}\label{internal-level-sets}
By  $(\tilde{\Sigma}^{v_i}_\ast,\tilde{\pi}^{v_i}_\ast),\, \ast=r,s$ we denote a typical leaf of the Poisson foliation  $\mathcal{\tilde{F}}^{v_i}_\ast\;$.  
\end{defn}

Ignoring (for a moment) the fact that the function $G^{v_i}_r$ is only defined on $T_{\gamma_{i-1},\gamma_i}$, we may consider  $\Pi_{\gamma_{i-1},\gamma_i}$ as the zero level set of $G^{v_i}_r$. Hence a typical leaf $\tilde{\Sigma}^{v_i}_r$ is diffeompric to   $\Pi_{\gamma_{i-1},\gamma_i}$ through a translation of type \eqref{translations}. Through this, diffeomorphism  $\Pi_{\gamma_{i-1},\gamma_i}$ secures a Poisson structure. Similarly, $L_{\gamma_{i-1},\gamma_i}(\Pi_{\gamma_{i-1},\gamma_i})$ gains a Poisson structure from $(\tilde{\Sigma}^{v_i}_s,\tilde{\pi}^{v_i}_s)$.

\begin{proposition}\label{prop:poincare-is-Poisson-vertex}
Let $\Pi_{\gamma_{i-1},\gamma_i}$ be equipped with the Poisson structure induced from $(\tilde{\Sigma}^{v_i}_s,\tilde{\pi}^{v_i}_s)$ via a translation of type \eqref{translations}, and $L_{\gamma_{i-1},\gamma_i}(\Pi_{\gamma_{i-1},\gamma_i})$ with the one induced from $(\tilde{\Sigma}^{v_i}_r,\tilde{\pi}^{v_i}_r)$ in a similar way.
Then $L_{\gamma_{i-1},\gamma_i}$ is Poisson map (see Figure~\ref{figure:illustration}).
\end{proposition}

\begin{proof}
We decompose $L_{\gamma_{i-1},\gamma_i}$ into three maps $P^{v_i}_1$, $P^{v_i}_2$ and $P^{v_i}_3$,  where $P^{v_i}_1$ and $P^{v_i}_3$   are the translations used to define the Poisson structures on  $\Pi_{\gamma_{i-1},\gamma_i}$ and $L_{\gamma_{i-1},\gamma_i}(\Pi_{\gamma_{i-1},\gamma_i})$, respectively, and 
$P^{v_i}_2$ is the Poincar\'e map from  $(\tilde{\Sigma}^{v_i}_r,\tilde{\pi}^{v_i}_r)$ to $(\tilde{\Sigma}^{v_i}_s,\tilde{\pi}^{v_i}_s)$. By the construction of these two sections, together with 
 Lemma~\ref{cosymplectic-foliations}, $P^{v_i}_2$  is a Poisson map, which ends the proof.   
\end{proof}

\begin{center}  
  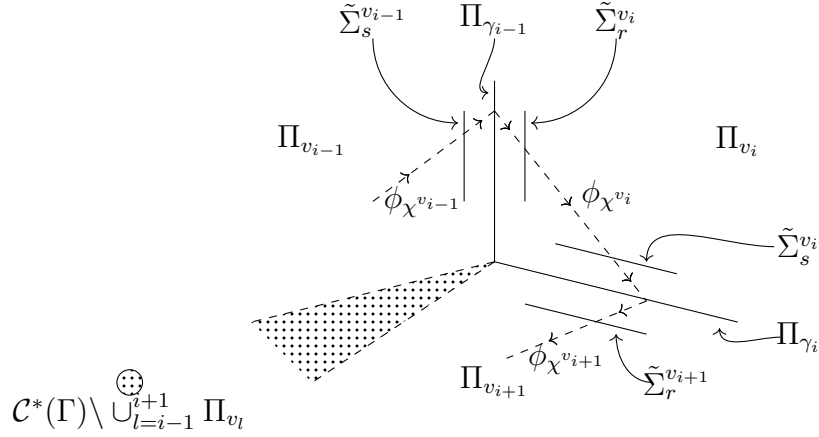
\begin{figure}[h]
\begin{tikzpicture}[scale=.8]
\tikzset{->-/.style={decoration={
  markings,
  mark=at position #1 with {\arrow{>}}},postaction={decorate}}}
\draw (7,3)--(7,0)--(11,-1);\draw[dashed] (7,0)--(4,-2);\draw[dashed] (7,0)--(3,-1);\fill[pattern=dots] (7,0)--(3,-1)--(4,-2)--(7,0);
\filldraw[pattern=dots] (1,-2) circle (6pt) node[below] {$\CC^\ast(\Gamma)\backslash \cup_{l=i-1}^{i+1} \Pi_{v_l}$};
\node at (7,-2) {$\Pi_{v_{i+1}}$};\node at (11,2) {$\Pi_{v_i}$}; \node at (4,2) {$\Pi_{v_{i-1}}$}; \node at (7,4) {$\Pi_{\gamma_{i-1}}$};\node at (12,-1.25) {$\Pi_{\gamma_i}$};
\node at (9,4) {$\tilde{\Sigma}^{v_i}_r$};\node at (5,4) {$\tilde{\Sigma}^{v_{i-1}}_s$};\node at (12,.25) {$\tilde{\Sigma}^{v_i}_s$};\node at (10,-2) {$\tilde{\Sigma}^{v_{i+1}}_r$};
\draw (7.5,2.5)--(7.5,1);\draw (6.5,2.5)--(6.5,1);\draw (8,.3)--(10,-.2);\draw (7.5,-.7)--(9.5,-1.2);

\draw[->] (7,3.7) to[out=270,in=180] (6.9,2.8);\draw[->] (9,3.7) to[out=270,in=0] (7.6,2.3);\draw[->] (5,3.7) to[out=270,in=180] (6.4,2.3);
\draw[->] (11.6,.25) to[out=180,in=90] (9.5,0);\draw[->] (9.6,-2) to[out=180,in=270] (9,-1.2);\draw[->] (11.7,-1.25) to[out=180,in=270] (10.7,-1.1);

\draw[->-=.5,->-=.1,->-=.9,dashed] (7,2.5) to (9.5,-.65);\node[right] at (8.25,1.1) {$\phi_{\chi^{v_i}}$};
\draw[->-=.3,->-=.9,dashed] (5,1) to (7,2.5);\node[right] at (5,1) {$\phi_{\chi^{v_{i-1}}}$};
\draw[->-=.2,->-=.7,dashed] (9.5,-.65) to (7.2,-1.6);\node at (8.2,-1.6) {$\phi_{\chi^{v_{i+1}}}$};

\end{tikzpicture}
\caption{\footnotesize{Illustration of Proposition~\ref{prop:poincare-is-Poisson-vertex}.}}
\label{figure:illustration}
\end{figure}
\end{center}

Notice that  $(\tilde{\Sigma}^{v_i}_\ast,\tilde{\pi}^{v_i}_\ast),\,\ast=r,s$ is a union of Poisson submanifolds equipped with Dirac bracket. We  describe now the matrix representative of this Dirac bracket.

\begin{lemma}\label{local-representation}
The matrix representative, in the coordinate system  $(y_{l})_{l\in [\alpha] \setminus \{j_{i\alpha}\}}$, of the Dirac bracket generated in $\inter(\Pi_{v_i})$ by the second class constrains $\eta_{v_i}$ and $G^{v_i}_r,$ is 
\begin{equation}\label{representing-matrix-r}
(\pi^{v_i}_r)^\sharp=B_{v_i}-C_{(v_i,r)}\;,
\end{equation}
where $C_{(v_i,r)}=[C^{\alpha,\beta}_{(v_i,r)}]_{\alpha,\beta}$ with 
$C^{\alpha,\beta}_{v_i,r}=[c_{lf}(\alpha,\beta,v_1,r)]_{( l , f)\in \{[\alpha]\setminus \{j_{i\alpha}\}\}\times \{[\beta]\setminus \{j_{i\beta}\}\} }$
and
\[c_{lf}(\alpha,\beta,v_i,r)=\frac{1}{\chi^{v_i}_{r_i}}\left((\chi^{v_i})^\alpha_{l}\, b^{\xi_{i-1},\beta}_{r_if}+b^{\alpha,\xi_{i-1}}_{lr_i}\,(\chi^{v_i})^\beta_{f}\right).\]
In the matrix $(\pi^{v_i}_r)^\sharp$ the $r_i^{\rm th}$ line and column are null. Removing these line and column one obtains the matrix representative, in the coordinate system  obtained by removing $y_{r_i}$ from  $(y_{l})_{l\in [\alpha] \setminus \{j_{i\alpha}\}}$, of the Poisson structure $\tilde{\pi}^{v_i}_r$ on $\tilde{\Sigma}^{v_i}_r$. Similarly, for the second class constrains $\eta_{v_i}$ and $G^{v_i}_s$ we have 
\begin{equation}\label{representing-matrix-s}
(\pi^{v_i}_s)^\sharp=B_{v_i}-C_{(v_i,s)}\;,
\end{equation}
where 
\[c_{lf}(\alpha,\beta,v_i,s)=\frac{1}{\chi^{v_i}_{s_i}}\left((\chi^{v_i})^\alpha_{l}\, b^{\xi_i,\beta}_{s_if}+b^{\alpha,\xi_i}_{ls_i}\,(\chi^{v_i})^\beta_{f}\right),\]
and removing the $s_i^{\rm th}$ line and column yields the matrix representative, in the coordinate system  obtained by removing $y_{s_i}$ from  $(y_l)_{l\in [\alpha] \setminus \{j_{i\alpha}\}}$, of the Poisson structure $\tilde{\pi}^{v_i}_s$ on $\tilde{\Sigma}^{v_i}_s$ . 
\end{lemma}

\begin{proof}
By definition, 

$$(\pi^{v_i}_r)^\sharp=\left[\, \{y^\alpha_l,y^\beta_f\} \,\right]_{( l , f)\in \{[\alpha]\setminus \{j_{i\alpha}\}\}\times \{[\beta]\setminus \{j_{i\beta}\}\} }.$$
So we need to calculate 
\begin{equation}
\begin{bmatrix}\{y^\alpha_{l},\eta_{n_i}\}&\{y^\alpha_{ l},G^{v_i}_r\}\end{bmatrix}\begin{bmatrix}0&\{\eta_{v_i},G^{v_i}_r\}\\\{G^{v_i}_r,\eta_{n_i}\}&0\end{bmatrix}^{-1}\begin{bmatrix}\{\eta_{n_i},y^\beta_{f}\}\\\{G^{v_i}_r,y^\beta_{f}\}\end{bmatrix},\
\end{equation}
see the definition of $\{.,.\}_{\rm Dirac}$  in~\eqref{eq:dirac-bracket}. Reminding that
\[\{\eta_{v_i},y^\alpha_{l}\}=(\chi^{v_i})^\alpha_{l}\quad\mbox{and} \quad\{y^\alpha_{l},y^\beta_{f}\}=b^{\alpha,\beta}_{lf},\]
  together with  a simple calculation, yields~\eqref{representing-matrix-r}. The $r_i^{\rm th}$ line and column are zero simply because, by definition, $G^{v_i}_r=y_{r_i}$ is  a Casimir  of the Dirac bracket. Note that the representative matrix $(\pi^{v_i}_r)^\sharp$ is with respect to the coordinate system as $(y_{l})_{l\in [\alpha] \setminus \{j_{i\alpha}\}}$ of $\Pi_{v_i}$, and by omitting the component $y_{r_i}$ from this coordinate system one obtains a coordinate system on  $\tilde{\Sigma}^{v_i}_r$. Therefore, removing the null $r_i^{\rm th}$ line and column yields the  representative matrix of $\tilde{\pi}^{v_i}_r$ with respect to the obtained coordinate. The same reasoning holds for $(\pi^{v_i}_s)^\sharp$. 
\end{proof}

We now extend Proposition \ref{prop:poincare-is-Poisson-vertex} to the whole heteroclinc path $\xi$. Our main result is the following.
\begin{theorem}\label{main theorem}
Let  
\begin{equation}\label{eq:path1}
\xi: v_0\buildrel\gamma_0 \over\longrightarrow  v_1 \buildrel\gamma_1 \over\longrightarrow  v_2 \longrightarrow  \ldots \longrightarrow  v_m \buildrel\gamma_m \over\longrightarrow v_{m+1}
\end{equation}
be a heteroclinic path. Then  for every $i=1,\ldots,m$, the Poisson structures induced on the intersection
\begin{equation}\label{eq:intersection}
L_{\gamma_{(i-2)},\gamma_{(i-1)}}(\Pi_{\gamma_{(i-2)},\gamma_{(i_1)}})\cap \Pi_{\gamma_{(i-1)},\gamma_{i}},
\end{equation}
from Poisson submanifolds  $(\tilde{\Sigma}^{v_{i-1}}_s,\tilde{\pi}^{v_{i-1}}_s)$ and $(\tilde{\Sigma}^{v_i}_r,\tilde{\pi}^{v_i}_r)$ is the same. Consequently, the skeleton flow map of $\chi$ along  $\xi$ (see Definition~\ref{poincaremap}),
$$\pi_\xi:=L_{\gamma_{m-1},\gamma_m}\circ\ldots\circ L_{\gamma_0,\gamma_1}\; ,$$ 
is a Poisson map w.r.t. the Poisson structures induced by $(\tilde{\Sigma}^{v_0}_r,\tilde{\pi}^{v_0}_r)$ and $(\tilde{\Sigma}^{v_m}_s,\tilde{\pi}^{v_m}_s)$ on its domain and range, respectively. 
\end{theorem}

Considering the segment 
\[\ldots \longrightarrow  v_{i-2} \buildrel\gamma_{i-2} \over\longrightarrow  v_{i-1} \buildrel\gamma_{i-1} \over\longrightarrow  v_i \buildrel\gamma_i \over\longrightarrow  v_{i+1} \longrightarrow \ldots,
\]
the key point is to show that the Poisson structure induced from $(\tilde{\Sigma}^{v_{i-1}}_s,\tilde{\pi}^{v_{i-1}}_s)$ on $L_{\gamma_{(i-2)},\gamma_{(i-1)}}(\Pi_{\gamma_{(i-2)},\gamma_{(i-1)}})$ and the one induced from $(\tilde{\Sigma}^{v_i}_r,\tilde{\pi}^{v_i}_r)$ on $\Pi_{\gamma_{(i-1)}\gamma_i}$, match on the intersection~\eqref{eq:intersection} 
(see Figure \ref{figure:illustration}). To prove Theorem~\ref{main theorem} we need to state and prove two preliminary lemmas regarding this key point.  
  
The two sectors $\Pi_{v_{i-1}}$ and $\Pi_{v_i}$ are only different in the group $\xi_{i-1}$, where  $y_{r_i}=0$ for the elements of $\Pi_{v_{i-1}}$ and 
$y_{s_{i-1}}=0$ for the elements $\Pi_{v_i}$.  Let $P_{v_{i-1},v_i}:{\rm int}(\Pi_{v_{i-1}})\to{\rm int}(\Pi_{v_i})$ be the diffeomorphism of the form
 $$T_{i-1,i}\circ (P^1_{v_{i-1},v_i}\times\ldots\times P^p_{v_{i-1},v_i}),$$
where:
\begin{itemize}
\item[1)] For $\beta\neq \xi_{i-1}$ the associated component $P^\beta_{v_{i-1},v_i}$ is the identity map;
\item[2)] For any $l\in [ \xi_{i-1}] \setminus \{s_{i-1}\}$ 
 \[(P^{\xi_{i-1}}_{v_{i-1},v_i}(y))^{\xi_{i-1}}_{l}= \left\{\begin{array}{ccc}
 y^{\xi_{i-1}}_{ l}-y_{s_{i-1}}&\mbox{if}&l\neq r_i\\
 -y_{s_{i-1}}&\mbox{if}&l= r_i
 \end{array}\right. ;\]
 
 \item[3)] 
For the following notation to be consistent,  without loss of generality we assume that $s_{i-1}=j_{i\xi_{i-1}}<r_i=j_{(i-1)\xi_{i-1)}}$. 
Notice that for any given point $y\in \tilde{\Sigma}_{s_{i-1}}=(G^{v_{i-1}}_s)^{-1}(c)$ the map  $P^{\xi_{i-1}}_{v_{i-1},v_i}$ acts on the component $\xi_{i-1}$ as
\begin{align*}
y^{\xi_{i-1}}&=(y^{\xi_{i-1}}_1,\ldots,c,\ldots,\cancel{y_{r_i}},\ldots,
y^{\xi_{i-1}}_{n_{\xi_{i-1}}})\mapsto\\
& (y^{\xi_{i-1}}_1-c,\ldots,\cancel{y_{s_{i-1}}},\ldots,-c,...,y^{\xi_{i-1}}_{n_{\xi_{i-1}}}-c),
\end{align*}
where the notation $\cancel{y_{r_i}}$ means that the entry $y_{r_i}$ is missing in the corresponding vector.

The image point is not in $\tilde{\Sigma}^{v_{i-1}}_r=(G^{v_i}_r)^{-1}(c^\prime)\subset \Pi_{v_i}$. 
However composing with the translation 
$$T_{i-1,i}(y):=y+(\bar{0},\ldots,(0,\ldots,\underset{\overset{\uparrow}{=c+c^\prime}}{t_{(i-1)i}},\ldots,0),\ldots,\bar{0}),$$ 
we get $$P_{v_{i-1},v_i}(\tilde{\Sigma}^{v_{i-1}}_s)=\tilde{\Sigma}^{v_i}_r.$$
\end{itemize}

We restrict  the diffeomorphism $P_{v_{i-1},v_i}$ to an open set  $U^{i-1}_s$ around $\tilde{\Sigma}^{v_{i-1}}_s$ to get 
\[P_{v_{i-1},v_i}:U^{i-1}_s\to U^i_r,\]
where $U^i_r$ is an open set around  $\tilde{\Sigma}^{v_i}_r$.
\begin{lemma}\label{lemma:ambient-poisson}
The diffeomorphism
$$
P_{v_{i-1},v_i}:(U^{i-1}_s,B_{v_{i-1}})\to (U^i_r,B_{v_i})
$$ 
is Poisson, i.e. $P_{v_{l-1},v_l}$ preserves the ambient Poisson structure. 
\end{lemma}
\begin{proof}
A simple  calculation shows that $(d P^{\xi_{i-1}}_{v_{i-1},v_i})E^{\xi_{i-1}}_{v_{i-1}}=E^{\xi_{i-1}}_{v_i}$. 
To give the reader an idea, let $n_{\xi_{i-1}}=5, s_{i-1}=2$ and $r_i=4$ then
\[P^{\xi_{i-1}}_{v_{i-1},v_i}(y_1,y_2,y_3,y_5)=(y_1-y_2,y_3-y_2,-y_2,y_5-y_2)\]  
and 
\[\underbrace{\begin{bmatrix}1&-1&0&0\\0&-1&1&0\\0&-1&0&0\\0&-1&0&1\end{bmatrix}}_{d P_{v_{i-1},v_i}}
\underbrace{\begin{bmatrix}-1&0&0&1&0\\0&-1&0&1&0\\0&0&-1&1&0\\0&0&0&1&-1\end{bmatrix}}_{E^{\xi_{i-1}}_{v_{i-1}}}=
\underbrace{\begin{bmatrix}-1&1&0&0&0\\0&1&-1&0&0\\0&1&0&-1&0\\0&1&0&0&-1\end{bmatrix}}_{E^{\xi_{i-1}}_{v_i}}\]
Since for $\beta\neq\xi_{i-1}$ the component  $P^\beta_{v_{i-1},v_i}$ is the identity map we get $(d P_{v_{i-1},v_i})E_{v_{i-1}}=E_{v_i}$. This fact together with~\eqref{eq:Poissoncondition2} and  the definitions of $B_{v_{i-1}},B_{v_i}$ (see Lemma~\ref{rescale:Poisson-from}) finishes the proof.
\end{proof}

\begin{lemma}\label{P-preserves-constraints}
For the diffeomorphism $P_{v_{i-1},v_i}$ we have that:
\begin{itemize}
\item[1)] $G^{v_i}_r\circ P_{v_{i-1},v_i}=-G^{v_{i-1}}_s\circ T_{i-1,i}\;;$
\item[2)] $\eta_{v_i} \circ  P_{v_{i-1},v_i}=\eta_{v_{i-1}}-\lambda_{\xi_{i-1}}q^{\xi_{i-1}}_{r_i}G^{v_{i-1}}_s+\lambda_{\xi_{i-1}} q^{\xi_{i-1}}_{r_i} t_{i-1,i}.$
\end{itemize}
\end{lemma}

\begin{proof}
The first equality is trivial since,  for any $y\in\Pi_{v_{i-1}}$, the $r_i^{th}$ component of $G^{v_i}_r\circ P_{v_{i-1},v_i}(y)$ is $-y_{s_{i-1}}+t_{(i-1)i}$. 
For the second equality we have  
\begin{align*}
\eta_{v_i} \circ & P_{v_{i-1},v_i}(y)=(\sum_{\beta\neq\xi_{i-1}}\sum_{l\in [\beta]\setminus \{j_{i\beta}\}} \lambda_\beta q^{\beta}_{l} y_{l}^{\beta})\\&+(\lambda_{\xi_{i-1}} \sum_{\scriptscriptstyle {{f}\in [\xi_{i-1}]\setminus \{s_{i-1},r_i\}}}q^{\xi_{i-1}}_{f} (y^{\xi_{i-1}}_{f}-y_{s_{i-1}}))+\lambda_{\xi_{i-1}} q^{\xi_{i-1}}_{r_i} (-y_{s_{i-1}}+t_{i-1,i})\\
=&(\sum_{\beta\neq\xi_{i-1}}\sum_{l\in [\beta]\setminus \{j_{i\beta}\}} \lambda_\beta q^{\beta}_{l} y_{l}^{\beta})
+\lambda_{\xi_{i-1}}(\sum_{\scriptscriptstyle {f\in [\xi_{i-1}]\setminus \{s_{i-1},r_i\}}}q^{\xi_{i-1}}_{f} (y^{\xi_{i-1}}_{f}))
\\&-\lambda_{\xi_{i-1}}y_{s_{i-1}}\sum_{\scriptscriptstyle {f\in [\xi_{i-1}]\setminus \{s_{i-1}\}}} q^{\xi_{i-1}}_{f}
+\lambda_{\xi_{i-1}} q^{\xi_{i-1}}_{r_i} t_{i-1,i}.
\end{align*}
Then,  using the fact that $\sum_{\scriptscriptstyle {f\in [\xi_{i-1}]\setminus \{s_{i-1}\}}} q^{\xi_{i-1}}_{f}=q_{s_{i-1}}-1$ we get 
 \begin{align*}
\eta_{v_i} \circ & P_{v_{i-1},v_i}(y)=\sum_{l\in [\beta]\setminus \{j_{i\beta}\}} \lambda_\beta q^{\beta}_{l} y_{{l}}^{\beta}+\lambda_{\xi_{i-1}} (q^{\xi_{i-1}}_{r_i} t_{i-1,i}-q^{\xi_{i-1}}_{r_i} y_{s_{i-1}})\\
&=\eta_{v_{i-1}}(y)-\lambda_{\xi_{i-1}}q^{\xi_{i-1}}_{r_i} y_{s_{i-1}}+\lambda_{\xi_{i-1}} q^{\xi_{i-1}}_{r_i} t_{i-1,i}.
\end{align*}
\end{proof}

\begin{proof}[Proof of Theorem~\ref{main theorem}:]
By Lemma~\ref{lemma:ambient-poisson} 
\[\{\eta_{v_i} \circ  P_{v_{i-1},v_i}, G^{v_i}_r\circ P_{v_{i-1},v_i}\}_{\Pi_{v_{i-1}}}=\{\eta_{v_i}, G^{v_i}_r\}_{\Pi_{v_i}}.\]
Since $G^{v_i}_r, \eta_{v_i}$ are second class constraints,  then 
\begin{equation}\label{eq:GcircP-etacircP}
\eta_{v_i} \circ  P_{v_{i-1},v_i}\quad\mbox{and}\quad G^{v_i}_r\circ P_{v_{i-1},v_i}
\end{equation}
are also second class constraints. Considering the equalities obtained in Lemma~\ref{P-preserves-constraints}, this fact can be obtained by direct calculations and  $G^{v_{i-1}}_s, \eta_{v_{i-1}}$ being second class constraints. Furthermore, the Dirac structure on $\Pi_{v_{i-1}}$ generated by the second class constraints $\{\eta_{v_{i-1}}, G^{v_{i-1}}_s\}$ is the same as the one  generated by~\eqref{eq:GcircP-etacircP}. To see this, note that the foliation constituted  from the level sets of  $\{\eta_{v_{i-1}}, G^{v_{i-1}}_s\}$ is the same as the one made up from the level set of the constraints~\eqref{eq:GcircP-etacircP}. 
Also,  Dirac bracket (see~\eqref{eq:dirac-bracket}) defined by them is the same, since the second term in Definition~\eqref{eq:dirac-bracket} is the same whether it is computed  using the constraints  $\{\eta_{v_{i-1}}, G^{v_{i-1}}_s\}$ or the constraints~\eqref{eq:GcircP-etacircP}. 
Simply compare the following equations  
 \begin{equation*}
\begin{bmatrix}\{f, \eta_{v_{i-1}}\}\\\{f,G^{v_{i-1}}_s\}\end{bmatrix}^t\begin{bmatrix}0&\{ \eta_{v_{i-1}},G^{v_{i-1}}_s\}\\\{G^{v_{i-1}}_s, \eta_{v_{i-1}},\}&0\end{bmatrix}^{-1}\begin{bmatrix}\{ \eta_{v_{i-1}},g\}\\\{G^{v_{i-1}}_s,g\}\end{bmatrix},
\end{equation*}

 \begin{equation*}
{\scriptscriptstyle \begin{bmatrix}\{f,\eta_{v_{i-1}}-aG^{v_{i-1}}_s\}\\\{f,-G^{v_{i-1}}_s\}\end{bmatrix}^t\begin{bmatrix}0&\{\eta_{v_{i-1}},-G^{v_{i-1}}_s\}\label{eq:check2}\\\{-G^{v_{i-1}}_s,\eta_{v_{i-1}}\}&0\end{bmatrix}^{-1}\begin{bmatrix}\{\eta_{v_{i-1}}-aG^{v_{i-1}}_s,g\}\\\{-G^{v_{i-1}}_s,g\}\end{bmatrix},}
\end{equation*}
where $a=\lambda_{\xi_{i-1}}q^{\xi_{i-1}}_{r_i}$. The constant terms are ignored and we used the fact that $\{G^{v_{i-1}}_s,a G^{v_{i-1}}_s\}=0$ to simplify the middle term in the second equation.  

 We conclude that the diffeomorphism $P_{v_{i-1},v_i}$, in addition to preserving the ambient Poisson structures, preserves the Dirac brackets as well, and consequently
$$(P_{v_{i-1},v_i})|_{\tilde{\Sigma}^{v_{i-1}}_s}:(\tilde{\Sigma}^{v_{i-1}}_s,\tilde{\pi}^{v_{i-1}}_s)\to(\tilde{\Sigma}^{v_i}_r,\tilde{\pi}^{v_i}_r).$$

Let $P_3^{v_{i-1}}$ and $P_1^{v_i}$ be the translations as defined in the proof of Proposition~\ref{prop:poincare-is-Poisson-vertex}.  The restriction map $(P_{v_{i-1},v_i})|_{\tilde{\Sigma}^{v_{i-1}}_s}$ is also a translation, so there exists a vector $K_{(i-1),i} $ such that the following diagram is commutative. 

\begin{center}  
  \begin{figure}[h]
\begin{tikzpicture}[scale=.8]
\tikzset{->-/.style={decoration={
  markings,
  mark=at position #1 with {\arrow{>}}},postaction={decorate}}}
  \node at (0,0) {$L_{\gamma_{(i-2)},\gamma_{(i-1)}}(\Pi_{\gamma_{(i-2)},\gamma_{(i_1)}})\cap \Pi_{\gamma_{(i-1)},\gamma_{i}}$};
   \node at (-4,-3) {$\tilde{\Sigma}^{v_{i-1}}_s $};
    \node at (4,-3) {$ \tilde{\Sigma}^{v_i}_r$};
    \draw[->] (-3.5,-2.5) to (-.5,-0.5);\draw[->]  (3.5,-2.5) to (0.5,-0.5) ;\draw[->] (-3.3,-3) to (3.5,-3);
    \node at (0,-3.5) {$(P_{v_{i-1},v_i})|_{\tilde{\Sigma}^{v_{i-1}}_s}+K_{(i-1),i}$};\node at (-2.6,-1.5) {$P^{v_{i-1}}_s$};\node at (2.6,-1.5) {$P^{v_i}_r$};
\end{tikzpicture}
\end{figure}
\end{center}
This shows that the Poisson structures coming from different sides of $\Pi_{\Gamma_{i-1},\Gamma_i}$ match and we can compose the Poisson map $L_{\Gamma_{l-1},\Gamma_l}$ for $l=1,\ldots,m+1$. This finishes the proof. 
 \end{proof}
 
For a given edge $v_{i-1}\buildrel\gamma_{i-1} \over\longrightarrow  v_i$ if there are more than one edge going out from the vertex $v_i$, say $\gamma_k$, with $k=1,2$, the $\Pi_{\gamma_{i-1}\gamma_k}$ are disjoint open subsets of $\Pi_{\gamma_{i-1}}$. Considering all these disjoint Poisson submanifold all together we can state the following result whose proof is immediate from the previous results.

 \begin{theorem}
 Let $\Bscr_S(\chi)$ denote the set of all $S$-branches of the skeleton vector field $\xi$ (see Definition~\ref{structural:set})
 and set $D_S:=\cup_{\xi\in\Bscr_S(\chi)}\Pi_\xi$ to be the open submanifold of
$$
\left( \Pi_S, \{ \cdot , \cdot \}_S \right):=\cup_{\gamma\in S}(\Pi_\gamma,\{ \cdot , \cdot \}_\gamma),
$$
with the same Poisson structure.
Then the \textit{skeleton flow  map} $\pi_S:(D_S,\{ \cdot , \cdot \}_S)\to(\Pi_S,\{ \cdot , \cdot \}_S))$ is Poisson.
\end{theorem}

%%%%%%%%%%%%%%%%%%%%%%%%%%%%%%%%%%%

\section{Example}
 \label{sec:example}

We will now present an example of a Hamiltonian polymatrix replicator system with a non trivial dimension.
This example was chosen to provide an illustration of the concepts and main results of this paper.
In particular it has a small structural set with a simple heteroclinic network.

\subsection{The fish example}

Consider the polymatrix replicator system defined by matrix
$$
A=\left(
\begin{array}{ccccccc}
 0 & 1 & 0 & 0 & 0 & 0 & -1 \\
 -1 & 0 & 1 & 0 & 0 & 0 & 0 \\
 0 & -1 & 0 & 1 & 0 & 0 & 0 \\
 0 & 0 & -1 & 0 & 1 & 0 & 0 \\
 0 & 0 & 0 & -1 & 0 & 0 & 1 \\
 0 & 0 & 0 & 0 & 0 & 0 & 0 \\
 1 & 0 & 0 & 0 & -1 & 0 & 0 \\
\end{array}
\right)\,.
$$

We denote by $X_A$ the vector field associated to this polymatrix replicator that is defined on the polytope
$$
\Gamma_{(5,2)}:=\Delta^4\times\Delta^1\,.
$$
%formal equilibrium
The point
$$ q= \left(\frac{1}{9},\frac{1}{3},\frac{1}{9},\frac{1}{3},\frac{1}{9},\frac{2}{3},\frac{1}{3}\right)\in\Gamma_{(5,2)} $$
satisfies
\begin{itemize}
   \item[(1)] $Aq=(0,0,0,0,0,0,0)$;
   \item[(2)] $q_1+q_2+q_3+q_4+q_5=1$ and $q_6+q_7=1$\,,
\end{itemize}
where $q_i$ stands for the $i$-th component of vector $q$, and hence is an equilibrium of $X_A$ (see Proposition~\ref{prop int equilibria}).
Since matrix $A$ is skew-symmetric, the associated polymatrix replicator is conservative (see Definition~\ref{CPR}).

The polytope $\Gamma_{(5,2)}$ has seven facets labeled by an index $j$ ranging from $1$ to $7$, and 
designated by $\sigma_1,\ldots, \sigma_7$.
The vertices of the  phase space  $\Gamma_{(5,2)}$  are also labeled by $i\in \{1,\dots,10\}$, and designated by $v_1,\ldots, v_{10}$, as described in Table~\ref{tbl:vertices}.

\begin{center}
\begin{tabular}{|c|c|} \toprule
Vertex  					& $\Gamma_{(5,2)}$  \\ \bottomrule \toprule
$v_1=(1,6)$ 		& $(1,0,0,0,0,1,0)$					\\ \midrule
$v_2=(1,7)$ 		& $(1,0,0,0,0,0,1)$  					\\ \midrule
$v_3=(2,6)$ 		& $(0,1,0,0,0,1,0)$					\\ \midrule
$v_4=(2,7)$ 		& $(0,1,0,0,0,0,1)$					\\ \midrule
$v_5=(3,6)$ 		& $(0,0,1,0,0,1,0)$					\\ \bottomrule
\end{tabular}
\qquad
\begin{tabular}{|c|c|} \toprule
Vertex  					& $\Gamma_{(5,2)}$  \\ \bottomrule \toprule
$v_6=(3,7)$ 		& $(0,0,1,0,0,0,1)$					\\ \midrule
$v_7=(4,6)$ 		& $(0,0,0,1,0,1,0)$					\\ \midrule
$v_8=(4,7)$ 		& $(0,0,0,1,0,0,1)$					\\ \midrule
$v_9=(5,6)$ 		& $(0,0,0,0,1,1,0)$					\\ \midrule
$v_{10}=(5,7)$ 	& $(0,0,0,0,1,0,1)$					\\ \bottomrule
\end{tabular}
\vspace{2mm}
        \captionof{table}{\small{Identification of the ten vertices of the polytope, $v_1, \dots , v_{10}$ in $\Gamma_{(5,2)}$.}}% Add 'table' caption
        \label{tbl:vertices}
\end{center}

The skeleton character $\chi_A$ of $X_A$ is displayed in Table~\ref{ex1:chars}.
(See Definition~\ref{skeletoncharacter} and Proposition~\ref{prop:skeletoncharacter}.)

\begin{table}[h] 
\centering
\begin{tabular}{c||rrrrrrr}
\,$\chi^v_\sigma$ & $\sigma_1$  & $\sigma_2$  & $\sigma_3$  & $\sigma_4$  & $\sigma_5$ & $\sigma_6$ & $\sigma_7$ \\
\hline \hline \\[-4mm]
\,$v_1$           &  $*$   &  $1$  &  $0$   &  $0$   &  $0$   &  $*$  &  $-1$   \\
% \hline \\[-4mm]
\,$v_2$           &  $*$   &  $0$   &  $-1$  &  $-1$   &  $-2$   &  $1$  &  $*$   \\
% \hline \\[-4mm]
\,$v_3$           &  $-1$  &  $*$   &  $1$   &  $0$  &  $0$   &  $*$  &  $0$   \\
% \hline \\[-4mm]
\,$v_4$           &  $0$   &  $*$   &  $1$   &  $0$   &  $-1$  &  $0$  &  $*$   \\
% \hline \\[-4mm]
\,$v_5$           &  $0$   &  $-1$   &  $*$  &  $1$   &  $0$   &  $*$  &  $0$  \\
% \hline \\[-4mm]
\,$v_6$           &  $1$   &  $-1$   &  $*$  &  $1$   &  $-1$   &  $0$  &  $*$  \\
% \hline \\[-4mm]
\,$v_7$           &  $0$   &  $0$   &  $-1$  &  $*$   &  $1$   &  $*$  &  $0$  \\
% \hline \\[-4mm]
\,$v_8$           &  $1$   &  $0$   &  $-1$  &  $*$   &  $0$   &  $0$  &  $*$  \\
% \hline \\[-4mm]
\,$v_9$           &  $0$   &  $0$   &  $0$  &  $-1$   &  $*$   &  $*$  &  $1$  \\
% \hline \\[-4mm]
\,$v_{10}$       &  $2$  &  $1$   &  $1$   &  $0$   &  $*$   &  $-1$  &  $*$   \\
\hline \hline
%\vspace{-.1cm}
\end{tabular}
\vspace{.2cm}
\caption{\footnotesize{The skeleton character $\chi_A$ of $X_A$, where the symbol $*$ in the $i$-th line and $j$-th column of the table means that the vertex $v_i$ does not belong to the facet $\sigma_j$ of the polytope $\Gamma_{(5,2)}$.}}
\label{ex1:chars}
\end{table}

The edges of $\Gamma_{(5,2)}$ are designated by $\gamma_1,\ldots, \gamma_{25}$,
according to Table~\ref{tbl:edges}, where we write $\gamma=(i,j)$ to mean that
$\gamma$ is an edge connecting the vertices $v_i$ and $v_j$.
This model has $25$ edges: $12$ neutral edges,
$$
\gamma_2, \gamma_3, \gamma_4, \gamma_7, \gamma_8, \gamma_{10}, \gamma_{12}, \gamma_{16}, \gamma_{17}, \gamma_{18}, \gamma_{16}, \gamma_{22}, \gamma_{25},
$$
and $13$ flowing-edges,
$$
\gamma_1, \gamma_5, \gamma_6 \gamma_9, \gamma_{11}, \gamma_{13}, \gamma_{14}, \gamma_{15}, \gamma_{19}, \gamma_{20}, \gamma_{21}, \gamma_{23}, \gamma_{24}.
$$
The flowing-edge directed graph of  $\chi_A$  is depicted in Figure~\ref{graph_ex_1}.

\begin{table}[h]
\centering
\begin{tabular}[c]{lllll}
\\
\hline
\\[-3mm]
$\gamma_1=(1, 2)$ 		& $\gamma_6=(3,1)$    	& $\gamma_{11}=(2,8)$   & $\gamma_{16}=(3,7)$	& $\gamma_{21}=(8,6)$ \\
$\gamma_2=(3,4)$ 		& $\gamma_7=(2,4)$    	& $\gamma_{12}=(1,9)$   	& $\gamma_{17}=(4,8)$	& $\gamma_{22}=(5,9)$ \\
$\gamma_3=(5,6)$ 		& $\gamma_8=(1,5)$    	& $\gamma_{13}=(2,10)$	& $\gamma_{18}=(3,9)$	& $\gamma_{23}=(6,10)$ \\
$\gamma_4=(7,8)$ 		& $\gamma_9=(2,6)$    	& $\gamma_{14}=(5,3)$   & $\gamma_{19}=(4,10)$	& $\gamma_{24}=(9,7)$ \\
$\gamma_5=(10,9)$		& $\gamma_{10}=(1,7)$ & $\gamma_{15}=(6,4)$ 	& $\gamma_{20}=(7,5)$	& $\gamma_{25}=(8,10)$
\hspace{-.27cm} \vspace{1mm} \\
\hline \vspace{-.2cm}
\end{tabular}
\caption{\footnotesize{Edge labels.}}\label{tbl:edges}
\end{table}

From this graph we can see that
$$ S=\{ \, \gamma_1=(1,2) \, \}  $$
is a  structural set for $\chi_A$ (see Definition~\ref{structural:set})
whose $S$-branches denoted by $\xi_1,\dots,\xi_5$ are displayed in Table~\ref{branch_table_ex1},
where we write $\xi_i=(j, k, l, \dots)$ to indicate that
$\xi_i$ is a path from vertex $v_j$ passing along vertices $v_k, v_l, \dots$\,.

\begin{figure}[h]
\begin{center}
\includefigure{width=10cm}{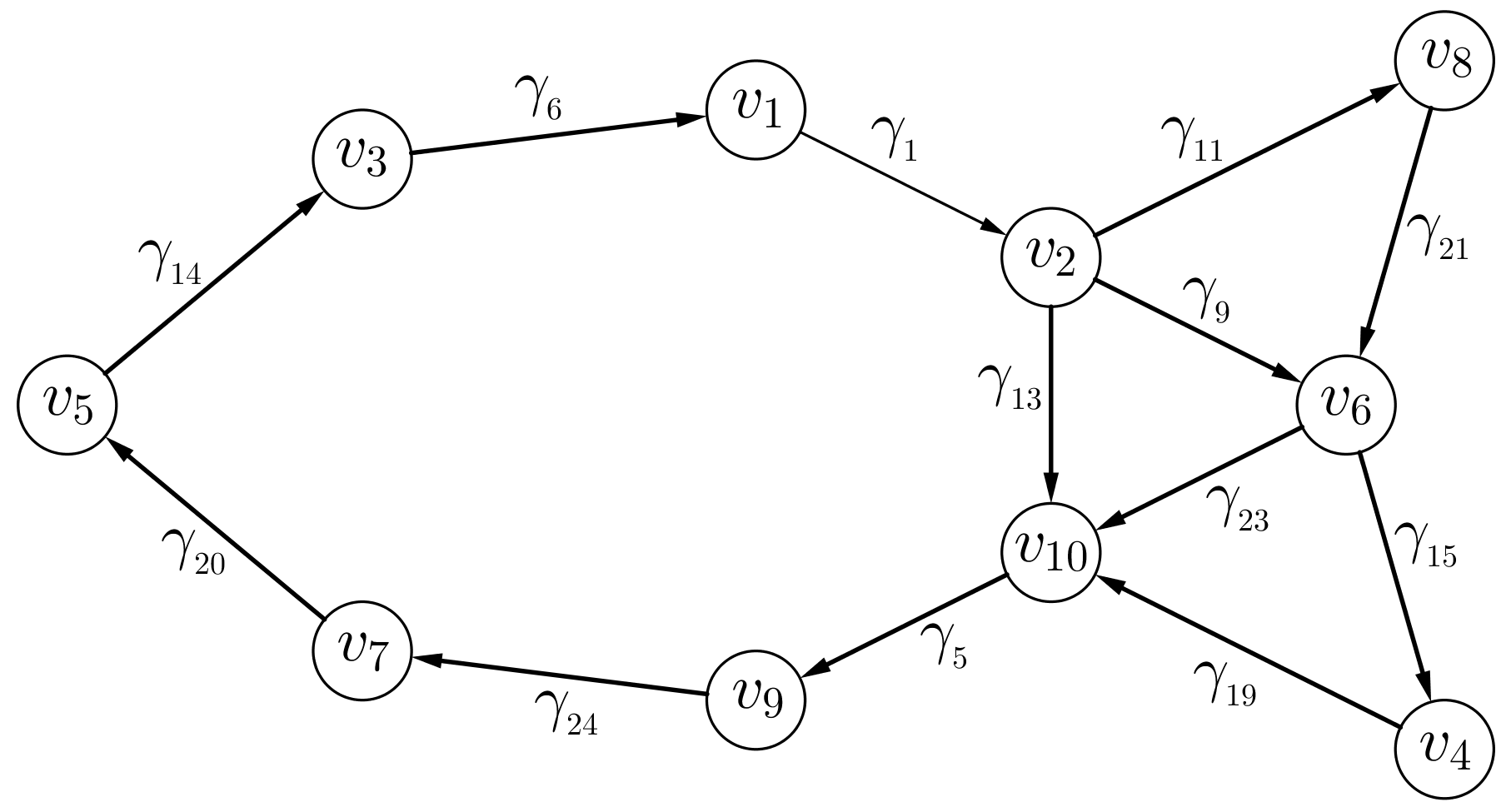}
\vspace{-.3cm}
\caption{\footnotesize{The oriented graph of $\chi_A$.}} \label{graph_ex_1}
\end{center} 
\end{figure}

\begin{table}[h]
\centering
\begin{tabular}{|c||l|}
\hline \\[-3mm]
From\textbackslash To & \multicolumn{1}{c|}{$\gamma_1=(1,2))$} \\[1mm]
\hline \hline
\\[-3mm]
\multirow{5}{*}{$\gamma_1=(1,2)$} 	& $\xi_1= (1, 2, 10, 9, 7, 5, 3, 1, 2)$ \vspace{1mm} \\
														& $\xi_2=(1, 2, 6, 10, 9, 7, 5, 3, 1, 2)$ \vspace{1mm} \\
  														& $\xi_3=(1, 2, 6, 4, 10, 9, 7, 5, 3, 1, 2)$ \vspace{1mm} \\
   														& $\xi_4=(1, 2, 8, 6, 10, 9, 7, 5, 3, 1, 2)$ \vspace{1mm} \\
    														& $\xi_5=(1, 2, 8, 6, 4, 10, 9, 7, 5, 3, 1, 2)$ \vspace{1mm} \\
\hline \hline
\end{tabular}
\vspace{.2cm}
\caption{\footnotesize{$S$-branches of  $\chi_A$.}} \label{branch_table_ex1}
\end{table}

Considering the vertex ${v_1}$, which has the incoming edge $v_3\buildrel\gamma_6 \over\longrightarrow {v_1}$ and the outgoing edge  ${v_1} \buildrel\gamma_1\over\longrightarrow v_2$, we will now illustrate Proposition~\ref{prop:poincare-is-Poisson-vertex}.

For $i=1,2,3$, the constant Poisson structures $B_{v_i}$  induced by asymptotic rescaling on each $\Pi_{v_i}$ (see Lemma~\ref{rescale:Poisson-from}) can be easily calculated:

$$
B_{v_1}=\left(
\begin{array}{ccccc}
 0 & 2 & 1 & 1 & 1 \\
 -2 & 0 & 1 & 0 & 1 \\
 -1 & -1 & 0 & 1 & 1 \\
 -1 & 0 & -1 & 0 & 2 \\
 -1 & -1 & -1 & -2 & 0 \\
\end{array}
\right),
\quad
B_{v_2}=\left(
\begin{array}{ccccc}
 0 & 2 & 1 & 1 & -1 \\
 -2 & 0 & 1 & 0 & -1 \\
 -1 & -1 & 0 & 1 & -1 \\
 -1 & 0 & -1 & 0 & -2 \\
 1 & 1 & 1 & 2 & 0 \\
\end{array}
\right)
$$
and
$$
B_{v_3}=\left(
\begin{array}{ccccc}
 0 & -2 & -1 & -1 & -1 \\
 2 & 0 & 2 & 1 & 0 \\
 1 & -2 & 0 & 1 & 0 \\
 1 & -1 & -1 & 0 & 1 \\
 1 & 0 & 0 & -1 & 0 \\
\end{array}
\right),
$$
and by~\eqref{representing-matrix-s} we get 

$$
(\pi^{v_1}_{{\rm Dirac}, 2})^\sharp =(\pi^{v_2}_{{\rm Dirac}, 0})^\sharp=\left(
\begin{array}{ccccc}
 0 & 1 & 0 & -1 & 0 \\
 -1 & 0 & 1 & 0 & 0 \\
 0 & -1 & 0 & 1 & 0 \\
 1 & 0 & -1 & 0 & 0 \\
 0 & 0 & 0 & 0 & 0 \\
\end{array}
\right)
$$
and
$$
(\pi^{v_3}_{{\rm Dirac}, 2})^\sharp=(\pi^{v_0}_{{\rm Dirac}, 1})^\sharp=\left(
\begin{array}{ccccc}
 0 & 0 & 0 & 0 & 0 \\
 0 & 0 & 1 & 0 & -1 \\
 0 & -1 & 0 & 1 & 0 \\
 0 & 0 & -1 & 0 & 1 \\
 0 & 1 & 0 & -1 & 0 \\
\end{array}
\right).
$$

The matrix $(\pi^{v_2}_{{\rm Dirac}, 0})^\sharp$ represents the Poisson structure on $\Pi_{\gamma_6}$ in the coordinates $(y_2,y_3,y_4,y_5,y_7)$.
 Notice that $y_2=0$ on $\Pi_{\gamma_6}$. Similarly, the matrix  $(\pi^{v_0}_{{\rm Dirac}, 1})^\sharp$ represents the Poisson structure on $\Pi_{\gamma_1}$ in the same coordinates $(y_2,y_3,y_4,y_5,y_7)$.
Notice again that $y_7=0$ on $\Pi_{\gamma_1}$.
 Now the matrix representative of $L_{\gamma_6\gamma_1}$ in the coordinates $(y_2,y_3,y_4,y_5,y_7)$ is 
 \[L_{\gamma_6\gamma_1}=\left(
\begin{array}{ccccc}
 0 & 0 & 0 & 0 & 1 \\
 0 & 1 & 0 & 0 & 0 \\
 0 & 0 & 1 & 0 & 0 \\
 0 & 0 & 0 & 1 & 0 \\
 0 & 0 & 0 & 0 & 0 \\
\end{array}
\right).\]
A simple calculation shows that 
\[L_{\gamma_6\gamma_1}\,(\pi^{v_2}_{{\rm Dirac}, 0})^\sharp\,(L_{\gamma_6\gamma_1})^t=(\pi^{v_0}_{{\rm Dirac}, 1})^\sharp,\]
confirming the fact that the asymptotic Poincar\'e map $L_{\gamma_6\gamma_1}$ is Poisson
(see~\eqref{eq:Poissoncondition2} in Definition~\ref{def:Poissonmap}).

Consider now  the subspaces of $\Rr^7$
$$
H=\left\{(x_1,\dots ,x_7)\in\Rr^7\,\,:\,\,\sum_{i=1}^5 x_i=1,\,\, \sum_{i=6}^7 x_i=1 \right\}
$$
and
$$
H_0=\left\{(x_1,\dots ,x_7)\in\Rr^7\,\,:\,\,\sum_{i=1}^5 x_i=0,\,\, \sum_{i=6}^7 x_i=0 \right\}.
$$

For the given matrix $A$, its null space $\Ker(A)$ has dimension $3$.
Take a non-zero vector $w\in \Ker(A)\cap H_0$.
For example,
$$
w=\left( -2, 3, -2, 3, -2, -3, 3 \right).
$$
The set of equilibria  
of the natural extension of $X_A$ to the affine hyperplane $H$
is
$$ \Eq(X_A)=\Ker(A)\cap H=\{q+tw \colon t\in\Rr\}\,. $$

The Hamiltonian of $X_A$ is the function $h_q:\Gamma_{(5,2)}\to\Rr$  
$$ h_q(x):=\sum_{i=1}^7 q_i\log x_i\,,$$
where $q_i$ is the $i$-th component of the equilibrium point $q$ (see Theorem~\ref{conservative:hamiltonian}).
Another  integral of motion of $X_A$ is the function $h_w:\Gamma_{(5,2)}\to\Rr$ 
$$ h_w(x):=\sum_{i=1}^7 w_i\log x_i\,,$$
where $w_i$ is the $i$-th component of $w$,
which is a Casimir of the underlying Poisson structure.

The skeletons  of $h_q$ and
$h_w$ are respectively
$\eta_q, \eta_w:\CC^\ast(\Gamma_{(5,2)})\to\Rr$,
$$ \eta_q(y):=\sum_{i=1}^7 q_iy_i \quad \textrm{and} \quad
\eta_w(y):=\sum_{i=1}^7 w_i y_i\, , $$
(see Proposition~\ref{asymptotic-constants}),
which we use to define $\eta:\CC^\ast(\Gamma_{(5,2)})\to\Rr^2$,
$$
\eta(y):=(\eta_q(y),\eta_w(y)).
$$

Consider the skeleton flow map  $\skPoin{\chi}{S}:\skDomPoin{\chi}{S} \to\Pi_{S}$ of $\chi_A$
(see Definition~\ref{def skeleton flow map}).
Notice that $\skDomPoin{\chi}{S}=\Pi_{\gamma_1}$, where 
by Proposition~\ref{partition},
$\Pi_{\gamma_1}=\bigcup_{i=1}^5\skDomPoin{\chi}{\xi_i}\pmod{0}$.
By Proposition~\ref{asymptotic-constants} the function  $\eta$ is invariant under $\pi_S$.
Moreover, the  skeleton flow map $\pi_S$ is  Hamiltonian with respect to a Poisson structure on the system of cross sections $\Pi_S$
(see Theorem~\ref{main theorem}).

For all $i=1,\dots,5$, the polyhedral cone $\Pi_{\xi_i}$ has dimension $4$.
Hence, each polytope $\Delta_{\xi_i,c}:=\Pi_{\xi_i}\cap \eta^{-1}(c)$   is a $2$-dimensional polygon.

\begin{remark}
We came from dimension $5$ to $2$.
This will happen for any other conservative polymatrix replicator with the same number of groups and the same number of strategies per group.
In fact when $n-p$ is odd, where $n$ is the total number of strategies in the population and $p$ is the number of groups,  we will have a minimum drop of $3$ dimensions.
The reason is that  a Poisson manifold with odd dimension (in this example is $5$) has at least one Casimir, and considering the transversal section we drop two dimensions from the symplectic part (not from the Casimir).
So in total  we drop a minimum of 3 dimensions.
If the original Poisson structure has more Casimirs, the invariant submanifolds yielded geometrically, are going to have even less dimensions, which is good as long as it not zero.
In the case of an even dimension, the drop will be at least of two dimensions.
\end{remark}

By invariance of   $\eta$,
the set $\Delta_{S,c}$ is also invariant under $\pi_S$. 
Consider now the restriction ${\pi_S}_{\vert \Delta_{S,c}}$ of $\skPoin{\chi}{S}$ to $\Delta_{S,c}$.
This is a piecewise affine area preserving map.
Figure~\ref{fig:plot_orbits_ex1} shows 
the domain  $\Delta_{S,c}$ and $20\,000$  iterates by  ${\pi_S}$  of a point in $\Delta_{S,c}$.
Following the itinerary of a random point we have picked the following
heteroclinic cycle consisting of $4$ $S$-branches
\begin{align*}\label{concat_path_ex1}
& \xi:=(\xi_4,\xi_1,\xi_3,\xi_4)\,.
\end{align*}

The   map $\pi_\xi$   is
represented by the matrix
$$
M_\xi =\left(
\begin{array}{ccccccc}
 0 & 0 & 0 & 0 & 0 & 0 & 0 \\
 1 & -1 & 1 & -\frac{13}{2} & 2 & -\frac{3}{2} & 0 \\
 1 & 0 & 1 & -1 & 1 & 2 & 0 \\
 -1 & 2 & -1 & \frac{15}{2} & -2 & \frac{5}{2} & 1 \\
 0 & 0 & 0 & 1 & 0 & 1 & 0 \\
 0 & 0 & 0 & 0 & 0 & 0 & 0 \\
 0 & 0 & 0 & 0 & 0 & 0 & 0 \\
\end{array}
\right)\,.
$$

The eigenvalues of $M_\xi$, besides $0$ and $1$ (with geometric multiplicity $3$ and $2$, respectively), are
$$ \lambda_u=5.31174..., \quad \textrm{and} \quad \lambda_s=\lambda_u^{-1}. $$

\begin{remark}
The determinant of  $(\pi^{v_2}_{{\rm Dirac}, 0})^\sharp$ is zero which means that the Poisson structure on
$\Pi_{\gamma_6}$ is non-degenerate.
So, $\Pi_6$ has a two dimensional symplectic foliation invariant under the asymptotic Poincar\'e map.
The leaf of this foliation are affine spaces parallel to the kernel of
\[ (\pi^{v_2}_{{\rm Dirac}, 0})^\sharp|_{\Pi_{\gamma_6}}=\left(
\begin{array}{cccc}
 0 & 1 & 0 & -1  \\
 -1 & 0 & 1 & 0  \\
 0 & -1 & 0 & 1 \\
 1 & 0 & -1 & 0 \\
 \end{array}
\right),\]
i.e. the set of the form 
$$
\{(q_3, q_4, q_5, q_7)+(s, t, -t, -s)\,|\, (q_3, q_4, q_5, q_7)\in \Pi_{\gamma_6}\,\, s,t\in\mathbb{R}\}\cap\Pi_{\gamma_6}.
$$
The restriction of the asymptotic Poincar\'e map to these leaves is a symplectic map.
One important consequence is that its eigenvalues are of the form $\lambda$ and $\frac{1}{\lambda}$.
\end{remark}

\begin{figure}[h]
\includefigure{width=9cm}{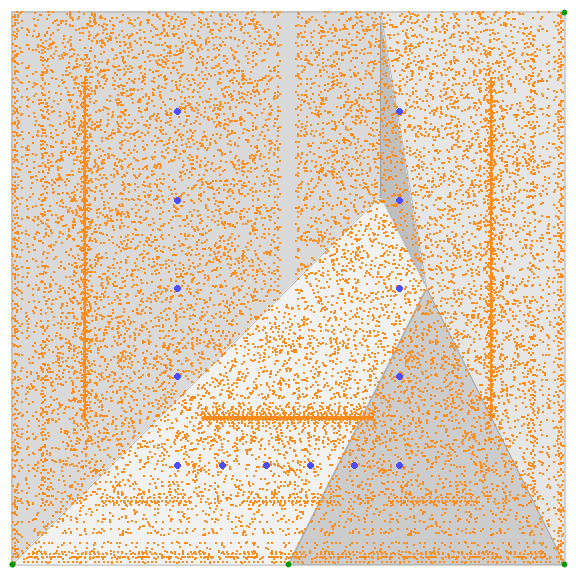}
\caption{\footnotesize{Plot of $20\,000$  iterates (in orange) by  ${\pi_S}$  of a point in $\Delta_{S,c}$, with $c=\left( \frac{1}{3},-0.5 \right)$,
the iterates of the periodic point $\mathbf{p_{\scriptsize{0}}}$ (in green) of the skeleton flow map $\skPoin{\chi}{S}$ with period $4$, and
the iterates of another periodic point of the skeleton flow map $\skPoin{\chi}{S}$ with period $14$ (in blue).}}\label{fig:plot_orbits_ex1}
\end{figure}

An eigenvector associated to the eigenvalue $1$ is
$$
\mathbf{p_{\scriptsize{0}}}=(0., 0.5, 1., 0., 0., 0., 0.)\,.
$$
We  have chosen  $c:=(c_1,c_2)=\left( \frac{1}{3},-0.5 \right)$ so that $\eta(\mathbf{p_{\scriptsize{0}}})=c$, \ie
$\mathbf{p_{\scriptsize{0}}}\in \Delta_{S,c}$.
In fact we have $\mathbf{p_{\scriptsize{0}}}\in\Delta_{\xi_1,c}\subset\Delta_{\gamma_1,c}$. 
Hence $\mathbf{p_{\scriptsize{0}}}$ is a periodic point of the skeleton flow map $\skPoin{\chi}{S}$ with period $4$
(whose iterates are represented by the green dots in Figure~\ref{fig:plot_orbits_ex1}).

Figure~\ref{fig:plot_orbits_ex1} also depicts
the polygons  $\Delta_{\xi_1,c}, \Delta_{\xi_2,c}, \Delta_{\xi_3,c}, \Delta_{\xi_4,c}, \Delta_{\xi_5,c}$ contained in $\Delta_{\gamma_1}$,
and the orbit of another periodic point of the skeleton flow map $\skPoin{\chi}{S}$ with period $14$
(represented by the blue dots in Figure~\ref{fig:plot_orbits_ex1}).

Following the procedure to analyze the dynamics in~\cite[Section $9$]{ADP2020} and using Theorem $8.7$ also in~\cite{ADP2020}  we could deduce the existence of
chaotic behavior for the flow of $X_A$ in some level set
$h_q^{-1}(c_1/\epsilon)\cap h_w^{-1}(c_2/\epsilon)$,
with the $c$ chosen above and for all small enough  $\epsilon>0$.

%%%%%%%%%%%%%%%%%%%%%%%%%%%%%%%%%%%%%%%%%%%%%%%%%%%%%%%%%%%%%%%%%%%%%%%%%%%%

\section*{Acknowledgements}

The first author was supported by mathematics department of UFMG. 
The second author was supported by
FCT - Funda\c{c}\~{a}o para a Ci\^{e}ncia e a Tecnologia,
under the projects UIDB/04561/2020 and UIDP/04561/2020.
The third author was supported by
FCT - Funda\c{c}\~{a}o para a Ci\^{e}ncia e a Tecnologia,
under the project UIDB/05069/2020.

%%%%%%%%%%%%%%%%%%%%%%%%%%%%%%%%%%%%%%%%%%%%%%%%%%%%%%%%%%%%%%%%%%%%%%%%%%%%

\nocite{*}
\bibliographystyle{amsplain}  

\end{document}